\documentclass[reqno,final,a4paper]{amsart}
\usepackage{color}
\usepackage{amsmath, amssymb, amsthm, thmtools}
\usepackage[colorlinks=true,allcolors=blue
]{hyperref}
\usepackage{mathrsfs}
\usepackage{mathtools}
\usepackage[linesnumbered,ruled,vlined]{algorithm2e}

\SetCommentSty{mycommfont}
\SetKwInOut{Input}{input}\SetKwInOut{Output}{output}

\usepackage[noadjust]{cite}
\usepackage[noabbrev,capitalize,nameinlink]{cleveref}
\crefname{equation}{}{}
\usepackage{fullpage}
\usepackage{graphics}
\usepackage{tikzscale}
\usepackage{pifont}
\usepackage{tikz}
\usetikzlibrary{fit, backgrounds, shapes, calc}
\usepackage{pgfplots}
\pgfplotsset{compat=1.12}
\usepackage{bbm}
\usepackage[T1]{fontenc}
\usetikzlibrary{arrows.meta}

\usepackage{environ}
\usepackage{framed}
\usepackage{url}
\usepackage{thmtools}
\usepackage{thm-restate}
\usepackage[noend]{algpseudocode}
\usepackage[labelfont=bf]{caption}
\usepackage{framed}
\usepackage[framemethod=tikz]{mdframed}
\usepackage{appendix}
\usepackage{graphicx}
\usepackage[textsize=tiny]{todonotes}
\usepackage{tcolorbox}
\usepackage{xcolor}
\usepackage[shortlabels]{enumitem}
\crefformat{enumi}{#2#1#3}
\crefrangeformat{enumi}{#3#1#4 to~#5#2#6}
\crefmultiformat{enumi}{#2#1#3}
{ and~#2#1#3}{, #2#1#3}{ and~#2#1#3}
\allowdisplaybreaks

\usepackage{ifthen}
\numberwithin{equation}{section}
\newtheorem{theorem}{Theorem}[section]
\newtheorem{proposition}[theorem]{Proposition}
\newtheorem{lemma}[theorem]{Lemma}
\newtheorem{claim}[theorem]{Claim}
\crefname{claim}{Claim}{Claims}

\newtheorem{corollary}[theorem]{Corollary}

\newtheorem*{question*}{Question}

\theoremstyle{definition}
\newtheorem{definition}[theorem]{Definition}

\newtheorem{question}[theorem]{Question}
\newtheorem*{definition*}{Definition}

\newtheorem*{fact*}{Fact}

\crefname{fact}{Fact}{Facts}

\theoremstyle{remark}

\renewcommand{\P}{\mathbb{P}}

\renewcommand{\Pr}{\mathbb{P}}
\newcommand{\E}{\mathbb{E}}
\newcommand{\eps}{\varepsilon}

\newcommand{\Var}{\operatorname{Var}}

\newcommand{\Bin}{\operatorname{Bin}}

\newcommand{\SMA}{\mathrm{SMALL}}
\newcommand{\MED}{\mathrm{MEDIUM}}
\newcommand{\LAR}{\mathrm{LARGE}}
\newcommand{\HC}{\mathrm{HC}}

\def\epsilon{\varepsilon}

\newenvironment{proofclaim}[1][Proof of claim]{\begin{proof}[#1]}{\end{proof}}

\let\originalleft\left
\let\originalright\right
\renewcommand{\left}{\mathopen{}\mathclose\bgroup\originalleft}
\renewcommand{\right}{\aftergroup\egroup\originalright}

\title{Hitting time for Hamilton cycles in pseudorandom graphs}

\author{Yaobin Chen}\address{Shanghai Center for Mathematical Sciences,~Fudan University,~Shanghai,~200438,~China.}\email{ybchen21@m.fudan.edu.cn}

\author{Yu Chen}\address{School~of~Mathematics~and~Statistics,~Beijing~Institute~of~Technology,~Beijing,~102488,~China.}\email{yu.chen2023@bit.edu.cn}

\author{Seonghyuk Im}\address{Center for AI and Natural Sciences, Korea Institute for Advanced Study (KIAS), Seoul, South Korea}\email{seonghyuk@kias.re.kr}

\author{Yiting Wang}\address{Institute of Science and Technology Austria,~Klosterneuburg,~3400,~Austria.}\email{yiting.wang@ist.ac.at}

\thanks{The third author is supported by the National Research Foundation of Korea (NRF) grant funded by the Korea government(MSIT) No. RS-2023-00210430, by the Institute for Basic Science (IBS-R029-C4), and by a KIAS individual Grant (AP 109501) at Korea Institute for Advanced Study. The fourth author is supported by the European Research Council (ERC), via grant agreements ``RANDSTRUCT'' No.\ 101076777.}

\date{\today}

\begin{document}
\begin{abstract}
Consider the random subgraph process on a base graph $G$ with $n$ vertices: we generate a sequence $\{G_t\}_{t=0}^{|E(G)|}$ by taking a uniformly random ordering of the edges of $G$ and then adding these edges one by one to the empty graph $G_0$ on the same vertex set. We prove that there is a constant $C > 0$ such that if $G$ is an $(n,d,\lambda)$-graph with $d/\lambda \ge C$, then with high probability, the hitting time for the appearance of a Hamilton cycle coincides with the hitting time for reaching minimum degree $2$. This resolves questions posed by Alon--Krivelevich in 2019 and by Frieze--Krivelevich in 2002.
As a consequence, we determine the sharp threshold for Hamilton cycles in $(n,d,\lambda)$-graphs with $d/\lambda\ge C$ for all $d$ sufficiently large. Lastly, we extend our result to the minimum degree $2k$ versus $k$ edge-disjoint Hamilton cycles setting for $k \leq c\cdot \min\{d,\log n\}$ where $c$ is a constant depending on $C$.
This advances on a question asked by Frieze.
\end{abstract}
\maketitle

\section{Introduction}
\subsection{Hitting time}
Consider a random graph process on an $n$-vertex graph $G$, defined as a sequence of nested graphs $\{G_t\}_{t=0}^{e(G)}$. The process begins with $G_0$ as the empty graph on the vertex set of $G$. For each $1 \le t \le e(G)$, the graph $G_t$ is obtained by adding a single edge to $G_{t-1}$, which is chosen uniformly at random from the set of edges in $G$ that are not currently present in $G_{t-1}$ (i.e., from $E(G) \setminus E(G_{t-1})$). A graph property $\mathcal{P}$ is called increasing if it is closed under the addition of edges; that is, if $H \in \mathcal{P}$ and $H \subseteq H'$, then $H' \in \mathcal{P}$. 
For an increasing non-empty graph property $\mathcal{P}$, the hitting time $\tau_\mathcal{P}$ is a random variable defined as the smallest index $t\in \mathbb N$ such that $G_t$ possesses $\mathcal{P}$. 
Formally:$$\tau_\mathcal{P} = \min \{t : G_t \in \mathcal{P}\} \,.$$
This represents the exact moment the process transitions from $G_{t-1} \notin \mathcal{P}$ to $G_t \in \mathcal{P}$.

Let $\mathcal H$ denote the property of Hamiltonicity, and let $\mathcal{D}_d$ denote the property of having a minimum degree of at least $d$. We denote the hitting time of $\mathcal{D}_d$ in a random graph process $G$ as $\tau_d(G)$ (or simply $\tau_d$ if the host graph $G$ is clear from context) and of $\mathcal H$ as $\tau_{\HC}$.
A celebrated result, proved independently by Ajtai, Komlós, and Szemerédi~\cite{Ajtai-Komlos-Szemeredi} and by Bollobás~\cite{Bollobas-hitting-time}, establishes that for a random graph process on the complete graph $G = K_n$, the following holds with high probability\footnote{We say that a sequence of events $\mathcal{E}_n$ holds with high probability if $\lim_{n \to \infty} \mathbb{P}(\mathcal{E}_n) =1$} (abbreviated as whp)
\begin{equation}\label{eqn:hitting-time}
    \tau_2(G) = \tau_{\HC}(G).
\end{equation}
This equality of hitting times demonstrates that, whp the graph becomes Hamiltonian at the precise moment it has minimum degree $2$.

It is natural to consider the Hamiltonicity of the random subgraph process of other host graphs. 
For instance, Bollob\'as and Kohayakawa~\cite{Bollobas-Kohayakawa} proved~\eqref{eqn:hitting-time} holds whp for $G = K_{n, n}$. 
Johansson~\cite{Johansson} proved~\eqref{eqn:hitting-time} holds whp for $G$ with minimum degree at least $\beta n$ for any constant $\beta > 1/2$.
Given those results, one may then ask what happens when the host graph is sparse. 
Condon, Espuny D\'{i}az, Gir\~{a}o, K\"{u}hn and Osthus~\cite{condon-et-al} proved~\eqref{eqn:hitting-time} holds whp for $n$-dimensional hypercubes.  
Moving away from highly structured host graphs like hypercubes, it is natural to investigate host graphs that are random or random-like. For Erd\H{o}s-R\'{e}nyi random graphs $G\sim G(n,m)$ with $m= (n\log n+n\log \log n + \omega(n))/2$ ($G \sim G(n, p)$ with $p = (\log n + \log\log n + \omega(1))/n$, respectively), we observed that whp (over the randomness of $G$ and of the random process), $G$ satisfies~\eqref{eqn:hitting-time}.
For completeness, we include a short proof of it in the appendix. 

Motivated by this result, it is natural to study the hitting time problem on pseudorandom graphs, which are families of deterministic graphs that behave similarly to random graphs. 
A prominent and extensively studied model of pseudorandom graphs is the spectral expanders, commonly referred to as $(n, d, \lambda)$-graphs. An $(n, d, \lambda)$-graph is an $n$-vertex $d$-regular graph where the second largest eigenvalue (in terms of absolute value) of its adjacency matrix is at most $\lambda$. 
The reason $(n,d,\lambda)$-graphs are considered pseudorandom is partially due to the expander mixing lemma, which states the parameter $\lambda$ dictates the edge distribution of $G$: a smaller $\lambda$ indicates that the edge distribution of $G$ resembles more closely to that of an Erd\H{o}s-R\'{e}nyi random graph $G(n, p)$ with $p = d/n$. 
For a more detailed introduction to this topic, we refer the reader to the excellent survey of Krivelevich and Sudakov~\cite{Krivelevich-Sudakov-survey}.

The study of the hitting time of Hamiltonicity in random graph processes on $(n,d,\lambda)$-graphs was initiated by Frieze and Krivelevich~\cite{Frieze-Krivelevich} in 2002. 
They established that if an $(n, d, \lambda)$-graph satisfies the spectral gap condition $\lambda = o(d^{5/2} / (n \log n)^{3/2})$, then whp, the hitting time of Hamiltonicity coincides with the hitting time of minimum degree at least $2$. 
However, this initial result carries a natural restriction on the degree. Because $\lambda = \Omega(\sqrt{d})$ for $d < (1-\varepsilon)n$ (see e.g.~\cite{Krivelevich-Sudakov-survey}), their theorem is only applicable when the degree is relatively large, specifically $d = \Omega((n\log n)^{3/4})$. 
This limitation prompted Frieze and Krivelevich to ask what is the weakest possible spectral gap requirement that preserves this hitting time property. They conjectured that, at least in the dense regime where $d = \Omega(n)$, the condition $\lambda = o(d)$ should suffice.
In 2019, Alon and Krivelevich~\cite{Alon-Krivelevich} relaxed the restriction on the spectral gap to $\lambda \le cd^2/n$ under a stronger assumption on $d$ that $d \ge C n \log \log n / \log n$ for some absolute constants $c, C>0$.
Despite these advancements, determining the optimal conditions on $d$ and $\lambda$ under which the hitting time result holds remains an open problem.

To understand the difficulty of this hitting time problem, note that
determining conditions for Hamiltonicity in $(n, d, \lambda)$-graphs has historically been a notoriously difficult question, even without the added complexity of random processes.
Krivelevich and Sudakov~\cite{krivelevich-Sudakov} made a well-known conjecture more than 20 years ago that an $(n,d,\lambda)$-graph is Hamiltonian if $d/\lambda$ is sufficiently large. 
After a series of works~\cite{krivelevich-Sudakov, Krivelevich-Sudakov-survey, allen2017powers, brandt2006global,hefetz2009hamilton, krivelevich2012number,Glock-MunhaCorreia-Sudakov}, a recent breakthrough by Dragani\'{c}, Montgomery, Munh\'{a} Correia, Pokrovskiy, and Sudakov~\cite{Draganic-et-al} finally settled this conjecture:

\begin{theorem}[\cite{Draganic-et-al}]\label{thm:draganic-et-al-ndlambda}
    There exists a constant $C_0 > 0$ such that if $G$ is an $(n,d,\lambda)$-graph with $\lambda\leq d/C_0$, then $G$ is Hamiltonian.
\end{theorem}
Building on the groundbreaking work~\cite{Draganic-et-al}, we prove that for any $(n, d, \lambda)$-graph, the hitting time of Hamiltonicity coincides with the hitting time of minimum degree at least $2$ whp, provided that the spectral ratio $d/\lambda$ is sufficiently large. As a consequence, we answer Frieze and Krivelevich's question in a strong form. 
\begin{theorem}\label{thm:hitting-time}
    There exist a constant $C > 0$ such that if $G$ is an $(n,d,\lambda)$-graph where $d/\lambda \geq C$, then we have $\tau_2(G) = \tau_{\HC}(G)$ whp.
\end{theorem}
By comparison, all previous results on hitting time hold only for $d$ polynomially large in $n$.
Our result only (implicitly) requires $d$ to be at least a large constant.

\subsection{Sharp thresholds}
Hitting time results are the strongest one can hope for since they pin down the exact moment the random subgraph becomes Hamiltonian. From there, we deduce a sharp threshold for Hamiltonicity in $(n,d,\lambda)$-graphs. Let $G_p$ denote the random subgraph of $G$ obtained by including each edge of $G$ independently with probability $p \in (0, 1)$. For an increasing graph property $\mathcal{P}$, we say that $p_0 \in (0, 1)$ is a sharp threshold for $\mathcal{P}$ if, for any constant $\varepsilon > 0$, the following holds whp: $G_{p} \notin \mathcal{P}$ when $p \le (1-\varepsilon)p_0$, and $G_{p} \in \mathcal{P}$ when $p \ge (1+\varepsilon)p_0$. 
The study of sharp threshold possesses a vast amount of literature and dates back to the birth of the study of the Erd\H{o}s-R\'{e}nyi random graphs. 
A few prominent examples that appear in every introductory course on random graphs include connectivity, Hamiltonicity, and containment of a subgraph of a fixed size. See, for instance, the classic textbooks on random graphs~\cite{bollobas2001random, janson2000random, frieze2015introduction}.
For an overview of this topic, we refer the reader to a recent survey by Perkins~\cite{perkins2025searching}.

As a corollary of \Cref{thm:hitting-time}, we determine the sharp threshold for Hamiltonicity in $(n, d, \lambda)$-graphs with $d/\lambda$ being sufficiently large.
\begin{corollary}\label{cor:sharp-threshold}
    There exist a constant $C > 0$ such that if $G$ is an $(n,d,\lambda)$-graph where $d/\lambda \geq C$, then the sharp threshold for the Hamiltonicity of $G$ is equal to 
    \[
        p_0 = \begin{cases}
            1 &\text{ if } d = o(\log n),\\
            1 - e^{-\log n/d}
           &\text{ if } d = \Omega(\log n).
        \end{cases}
    \]
    In particular, when $d = \omega(\log n)$, one can take $p_0$ to be $\log n/d$.
\end{corollary}
This result strengthens a recent result of the first and second authors, Han, and Zhao~\cite{Chen-Chen-Han-Zhao} from $d = \omega(\log n)$ to all $d$.
In fact, we obtain a more refined characterization of the threshold values, identifying the specific point at which the threshold behavior transitions. In particular, we prove that the threshold behaves analogously to the complete graph case only when $d = \Omega(\log^2 n)$.
\begin{corollary}\label{cor:sharp-threshold-refine}
     There exists a constant $C > 0$ such that if $G$ is an $(n,d,\lambda)$-graph where $d/\lambda \geq C$ then 
     \begin{itemize}
         \item If $d = o(\log n)$, then for any $\varepsilon>0$, $G_p$ is non-Hamiltonian for $p = 1 - (nd)^{-(1-\varepsilon)/(d-1)}$ whp and Hamiltonian for $p = 1 -(nd)^{-(1+\varepsilon)/(d-1)}$ whp.
         \item If $d = c\log n$, then for any $\varepsilon>0$, $G_p$ is non-Hamiltonian whp for $p = 1-e^{-\log n/d +\varepsilon}$ and Hamiltonian for $p = 1-e^{-\log n/d - \varepsilon}$ whp.
         \item If $d = \omega(\log n)$ and $d=o(\log^2 n)$, then for any $\varepsilon>0$, $G_p$ is non-Hamiltonian for $p = (\log n+ \log\log n)/d-(1+\varepsilon)\log^2 n/(2d^2)$ whp and Hamiltonian for $p =(\log n+ \log\log n)/d-(1-\varepsilon)\log^2 n/(2d^2)$ whp.
         \item If $d = \Omega(\log^2 n)$, then $G_p$ is non-Hamiltonian for $p = (\log n+ \log\log n-\omega(1))/d$ whp and Hamiltonian for $p =(\log n+ \log\log n+\omega(1))/d$ whp.
     \end{itemize}
\end{corollary}
We remark that in the third case, if $d=\omega(\log n)$ and $d=o(\log^2 n )$, the bound does not simplify to $p = (\log n+ \log\log n \pm \omega(1))/d$. 

\Cref{cor:sharp-threshold} and~\Cref{cor:sharp-threshold-refine} showed that $(n,d,\lambda)$ graphs with sufficiently large spectral gap are not only Hamiltonian but robustly so. These results add to a fruitful line of research on the robustness of graph properties.
See~\cite{sudakov2017robustness} for a systematic treatment of this topic and~\cite{joos2023robust,pham2022toolkit,allen2024robust,bastide2024random,han2025rainbow,chen2024thresholds} for some recent developments.

\subsection{$k$-edge disjoint Hamilton cycles}
Generalising the Hamiltonicity hitting time result, the study of $k$ edge-disjoint Hamilton cycles in random graphs has attracted considerable attention.
Let $k\mathrm{HC}$ denote the property of containing $k$ edge-disjoint Hamilton cycles.
A classical result by Bollob\'{a}s and Frieze~\cite{BF1985} established that if $G = K_n$, then for $k\in \mathbb N$, the hitting time for $k$ edge-disjoint Hamilton cycles coincides with the hitting time for minimum degree at least $2k$ whp, i.e., $\tau_{2k}(G) = \tau_{k\HC}(G)$ whp.
Condon, Espuny D\'{i}az, Gir\~{a}o, K\"{u}hn and Osthus~\cite{condon-et-al} extended the result to hypercubes. 
Alon and Krivelevich~\cite{Alon-Krivelevich} extended this result to $(n, d, \lambda)$-graphs with $d \ge C n \log \log n / \log n$ and $\lambda \le cd^2/n$ for some absolute constants $c, C>0$.
In Frieze's survey on Hamiltonicity, he posed the question of whether this result remains valid when $k$ is allowed to grow to infinity with $n$~\cite[Problem 26]{Frieze-survey}. 
Alon and Krivelevich~\cite{Alon-Krivelevich} confirmed it for $k = o(\log \log n)$\footnote{Alon and Krivelevich stated that their result extends to $k = o(\log \log n)$ in their concluding remarks, which was missed by Frieze when he raised the question.}. We advanced it further by proving it for $k = O(\log n)$.
\begin{theorem}\label{thm:k-edge-disjoint-HC}
    There exists a constant $C > 0$ such that if $G$ is an $(n,d,\lambda)$-graph with $d/\lambda \geq C$ and $\hat{d}=\min\{d,10\log n\}$, then we have $\tau_{2k}(G) = \tau_{k\HC}(G)$ whp for $k\le 10^{-10}\hat{d}/C^2$.
\end{theorem}
When $d \leq 10\log n$, one may take $k = 10^{-10}d/C^2$, which is optimal up to a constant factor in front of $d$. However, the theorem becomes less effective as one increases $d$ since now $d/k$ can be unbounded. The reason is that for larger $k = O(d)$, the nature of the problem changes drastically, and one essentially needs to solve a packing problem. We give a detailed explanation in the concluding remarks.

Now we give a brief proof overview, highlighting a technical result that might be of independent interest.

\subsection{Robust Hamiltonicity of $C$-expanders}
For a constant $C>0$, a graph $G$ is called a \emph{$C$-expander} if for every subset $S \subseteq V(G)$ with $|S| \leq n/2C$, we have $|N(S)| \geq C|S|$ and for every two disjoint subsets $S, T \subseteq V(G)$ with $|S|, |T| \geq n/2C$, there is at least one edge between them.~\Cref{thm:draganic-et-al-ndlambda} is in fact a direct corollary of the stronger result that Dragani\'{c}, Montgomery, Munh\'{a}-Correia, Pokrovskiy, and Sudakov~\cite{Draganic-et-al} proved:
\begin{theorem}[\cite{Draganic-et-al}]\label{thm:expander-hamiltonicity}
    There exists a constant $C_0 > 0$ such that if $G$ is an $n$-vertex $C_0$-expander, then $G$ is Hamiltonian.
\end{theorem}
While this theorem is a powerful tool for proving Hamiltonicity, it does not directly apply to the hitting time graph $G_{\tau_2}$ in our setting.
Indeed, for any host graph $G$, the hitting time graph $G_{\tau_2}$ contains a degree $2$ vertex, which implies that $G_{\tau_2}$ cannot be a $C$-expander for $C>2$.
Thus, we cannot directly apply the result of~\cite{Draganic-et-al} to conclude that $G_{\tau_2}$ is Hamiltonian even when $G = K_n$.
To overcome this issue, we show that a $C$-expander is Hamiltonian even when we force it to contain some specific edges in the Hamilton cycle, which may be of independent interest.
\begin{restatable}{theorem}{robust}\label{thm:robust-Hamiltonicity}
  There exists a constant $C>0$ such that the following holds for sufficiently large $n$.
    Let $G$ be an $n$-vertex $C$-expander and let $E_0 \subseteq E(G)$ be the set of edges with $|E_0| \leq n^{0.12}$ and any two edges in $E_0$ are at least distance $3$ apart in $G$.
    Then, $G$ contains a Hamilton cycle containing all edges in $E_0$.
\end{restatable}

Our proof strategy consists of two main parts. First, we analyse the structural properties of the graph $G_{\tau_2}$. We show that $G_{\tau_2}$ can be decomposed into a massive core with excellent expansion property (a $C$-expander) and a small, scattered set of low-degree vertices. Crucially, we show that these low-degree vertices are far apart, allowing us to ``clean-up'' the graph
by deleting these vertices and adding auxiliary edges between (some of) their neighbours. 
Then, by applying~\Cref{thm:robust-Hamiltonicity}, we can find a Hamilton cycle containing all the auxiliary edges in the auxiliary expander, which becomes a Hamilton cycle in the original graph after ``decontracting'' the auxiliary edges.
To prove~\Cref{thm:robust-Hamiltonicity}, we adapt the Pos\'{a} rotation and sorting network techniques inspired by~\cite{Draganic-et-al}.
The main difference for our setting is that we need to carefully control the edges in $E_0$ and ensure that they are not destroyed during the Pos\'{a} rotations.

To prove \Cref{thm:k-edge-disjoint-HC}, we adapt the proof of \Cref{thm:hitting-time} to show that after deleting any $k-1$ disjoint Hamilton cycles from the hitting time graph $G_{\tau_2}$, the remaining graph still contains a Hamilton cycle.
The main difference is that we need to show that the remaining graph still has good expansion properties after deleting $k-1$ Hamilton cycles, which requires showing an additional edge-expansion property instead of a single vertex-expansion property.

\Cref{cor:sharp-threshold-refine} follows from~\Cref{thm:hitting-time} by a first and second moment computation about the number of vertices with degree at most $1$.

\subsection{Paper organization} The paper is organized as follows. In~\Cref{sec:preliminaries}, we introduce notation and standard tools that are used throughout the paper. \Cref{subsec:hitting-time-properties} contains the proof of expansion properties of $G_{\tau_2}$ and the proof of~\Cref{thm:hitting-time} assuming \Cref{thm:robust-Hamiltonicity}. \Cref{thm:robust-Hamiltonicity} is proved in~\Cref{subsec:robust-hamiltonicity}.
In~\Cref{subsec:k-edge-disjoint-HC}, we discuss how the proof of~\Cref{thm:hitting-time} can be extended to~\Cref{thm:k-edge-disjoint-HC}.
\Cref{subsec:deducing-corollaries} contains the proof of~\Cref{cor:sharp-threshold} and~\Cref{cor:sharp-threshold-refine}. In~\Cref{sec:concluding-remarks}, we conclude the paper with a few remarks and open problems. At last, in~\Cref{sec:appendix}, we prove the hitting time result for random graphs mentioned in the introduction.

\subsection*{Acknowledgement}
This work was initiated while Yaobin Chen, Yu Chen, and Yiting Wang were visiting the ECOPRO group at the Institute for Basic Science (IBS) as participants in the ECOPRO Summer Student Research Program. The authors are grateful to the organizers, particularly Hong Liu, and the members of IBS for their hospitality and support. Additionally, the authors wish to thank Matthew Kwan and Patrick Morris for their valuable advice on the presentation of the paper.

\section{Preliminaries}\label{sec:preliminaries}
\subsection{Notation}
For a graph $G$, we let $e(G)$ denote the number of its edges. For any two disjoint sets of vertices $S,T\subseteq V(G)$, we write $e_{G}(S,T)$ for the number of edges between $S$ and $T$ in $G$. For a set of vertices $S$, we denote by $S^c = V(G)\setminus S$.
We denote the maximum degree of a graph $G$ by $\Delta(G)$. We denote internal neighborhood by $\Gamma(S) = \{u\in V(G): N(u)\cap S\neq \emptyset\}$ and the external neighborhood by $N(S) = \Gamma(S)\setminus S$. 
For a vertex $v\in V(G)$ and a vertex set $S\subseteq V(G)$, we denote the number of neighborhoods of $v$ in $S$ by deg$_S(v)$.

For two vertices $u,v\in V(G)$, the distance between them, denoted as $d_G(u,v)$, is the length of the shortest path between them in $G$, which is set to $\infty$ if there is no path between $u$ and $v$.
For two edges $e_1$ and $e_2$, the distance between them, denoted as $d_G(e_1,e_2)$, is defined as $\min_{(u,v)\in V(e_1)\times V(e_2)}d_G(u,v)$. 

We write $a = (1\pm \delta)b$ for $(1-\delta)b\leq a\leq (1+\delta)b$. 
The base of all logarithms and exponential functions is by default $e$. We denote the set of natural numbers as $\mathbb N = \{1,2,\dots\}$.

We use the usual asymptotic notation. Given two functions $f(n)$ and $g(n)$, if $f(n)/g(n) \rightarrow 0$ when $n\rightarrow \infty$, then we write $f(n) = o(g(n))$ and $g(n) = \omega(f(n))$. If there exists an absolute constant $C > 0$ for which $|f(n)|\leq C\cdot |g(n)|$ for all $n\in \mathbb N$, then we write $f(n) = O(g(n))$ and $g(n) = \Omega(f(n))$.

We say a random variable $X$ is \emph{stochastically dominated} by another random variable $Y$ if for every real number $t$, we have $\mathbb{P}(X\ge t)\le \mathbb{P}(Y\ge t)$. 

\subsection{Asymptotic equivalence of random graph models}\label{subsec:asymptotic-equiv}
Let $G$ be a graph on $n$ vertices with $N$ edges. Let $p\in [0,1]$ and $m\in [0,N]$ be an integer. We define $G_p$ to be a subgraph of $G$ that keeps each edge independently with probability $p$ and $G_m$ a uniformly random subgraph of $G$ with precisely $m$ edges.

We record the asymptotic equivalence between the model $G_p$ and $G_m$ for monotone properties for $p = m/N$: 
\begin{lemma}[{\cite[Corollary 1.16]{janson-book}}]\label{lem:asymptotic-equiv}
    Let $\delta>0$ be a constant.
    Let $G$ be a graph on $n$ vertices with $N$ edges with $N \to \infty$ as $n \to \infty$ and $\mathcal P$ be a monotone graph property.
    For any $0 \leq m \leq (1-\delta)N$, $p=m/N$, we have 
    \begin{itemize}
        \item If $G_p \in \mathcal{P}$ whp, then $G_m \in \mathcal{P}$ whp.
        \item If $G_p \notin \mathcal{P}$ whp, then $G_m \notin \mathcal{P}$ whp.
    \end{itemize}
\end{lemma}

\subsection{Expander mixing lemma}
For $(n,d,\lambda)$-graphs, one of the most useful tools is the expander mixing lemma. 
\begin{lemma}[{\cite{alon-chung}}]\label{lem:expander-mixing}
    Let $G$ be an $(n,d,\lambda)$-graph. Then, for any set $K\subseteq V(G)$ of size $k$, we have 
    \[
    |e_G(K) - dk^2/(2n)| \leq \lambda k/2\,.
    \]
    Moreover, for any two disjoint sets $K,L\subseteq V(G)$ of size $k,l$ respectively, we have
    \[
    |e_G(K,L) - dkl/n| \leq \lambda\sqrt{kl}\,.
    \]
\end{lemma}

\subsection{Large deviation}
For large deviations of binomial random variables, we will use the following simple bound, which can be deduced directly from the definition of the binomial distribution:
\begin{proposition}\label{prop:large-deviation}
    Let $p\in (0,1)$ and $X\sim \Bin(n,p)$. Then, for any $t\in \mathbb N$, we have 
    \[
        \P(X\geq t)\leq\binom{n}{t}p^t\,.
    \]
\end{proposition}

\subsection{Estimates}\label{subsec:estimates}
We use the following estimates that hold for all $n,t,s,p,q,k\in \mathbb N$ and $p\in [0,1]$. All of them are either standard or can be deduced straightforwardly from the definition of the binomial coefficients. 
\begin{enumerate}[itemsep = 3pt]
    \item $\binom{n}{t}\leq (en/t)^t$;\label{(1)}
    \item $\binom{n-s}{t}\leq \exp(-st/n)\cdot \binom{n}{t}$;\label{(2)}
    \item $1-p\leq \exp(-p)\le 1-p+p^2/2$;\label{(3)}
    \item $\sum_{i=i_0}^k\binom{q}{i}\binom{p}{k-i}\le\binom{q}{i_0}\binom{p+q-i_0}{k-i_0}$ if $s\le k$;\label{(4)}
    \item $\binom{n-s}{t-i}\leq \binom{n}{t} \cdot (t/n)^i\cdot \exp(-(s-i)(t-i)/n)$ if $i\le s$;\label{(5)}
    \item let $f(i)=\binom{s}{i}\binom{n-s}{t-i}$, then $f(i)\le 2f(i+1)$ if $i\leq st/(8n)$.\label{(6)}
\end{enumerate}

\section{Hitting time properties}\label{subsec:hitting-time-properties}
In this section, we establish various properties of the hitting time graph $G_{\tau_2}$, which are crucial for proving~\Cref{thm:hitting-time}.
We divide the proof into two cases. We call an $(n,d,\lambda)$ graph $G$ \emph{sparse} if $d\le 10\log n$, and \emph{dense} if $d> 10\log n$.
We begin with the sparse case.
\subsection{Sparse case}
The goal of this section is to establish the following: 
\begin{lemma}\label{lem:properties-sparse}
    There exists a constant $C$ such that the following statement holds. 
    Let $G$ be an $(n,d,\lambda)$-graph with $d\le 10\log n$ and $d/\lambda \geq C$.
    Then the subgraph at the hitting time of minimum degree 2, denoted as $G_\tau$, satisfies the following properties whp. 
    Let $\SMA(G_\tau) \coloneqq \{v\in V(G): \deg_{G_\tau}(v)\leq d/10^6\}$ and $C_1=C^{1/10}$. Then,
    \begin{enumerate}[label=\rm{(P\arabic*)}, itemsep = 3pt]
    \item $|\SMA(G_{\tau})|\le n^{0.11}$; \label{P1}
    \item $\forall~u,v\in \SMA(G_{\tau})$, we have $d_{G_{\tau}}(u,v)>4$;\label{P2}
    \item $\forall~U,V\subseteq V(G)$ disjoint and with $|U|,|V|\ge n/(4C_1^3)$, we have $e_{G_{\tau}}(U,V)\ge 1$; \label{P3}
    \item $\forall~W\subseteq V(G)\setminus \SMA(G_{\tau})$ with $|W|\le n/(2C_1)$, we have $|N_{G_{\tau}}(W)|\ge C_1|W|$. \label{P4}
\end{enumerate}
\end{lemma}
In~\ref{P3}, strengthening from $n/(2C_1)$ (which is what $C_1$-expander requires) to $n/(4C_1^3)$ is due to a technical reason. When applying~\Cref{lem:properties-sparse}, we only use it for $n/(4C_1)$. 
We begin by establishing a (crude) lower bound on the hitting time for Hamiltonicity:
\begin{proposition}\label{prop:bound-hitting-time-sparse}
Let $G$ be a $d$-regular graph on $n$ vertices with $d =O(\log n)$. Then whp, $\tau_2\geq (1-n^{-1/d})e(G)$.
\end{proposition}
\begin{proof}
    Let $t = (1-\eps) e(G)$ where $\eps = n^{-1/d}$.
    We claim that whp, $G_t$ has a vertex of degree $1$. This implies that whp $\tau_2\geq t$, as desired.
    Let $X_v$ be the indicator random variable for $\deg_{G_t}(v) = 1$ and $X = \sum_{v\in V(G)}X_v$. Observe that 
    \begin{align*}
        \E[X] = n\cdot \frac{d\binom{nd/2-d}{t-1}}{\binom{nd/2}{t}} & = nd \cdot \frac{(nd/2-d)!}{(nd/2-d-t+1)!(t-1)!}\frac{(nd/2-t)!t!}{(nd/2)!} \\
        & = nd\cdot \frac{t\cdot (nd/2 -d-t+2)\cdots (nd/2-t)}{(nd/2-d+1)\cdots (nd/2)}\geq (1+o(1))(1-\eps)nd \eps^{d-1} = \omega(1)\,,
    \end{align*}
    Indeed, if $d = O(1)$, then $\E[X] \geq n^{1/d} = \omega(1)$. If $d = \omega(1)$, then $\E[X]\geq d = \omega(1)$.
    Moreover, note that $\Var[X] = \sum_{(u,v)\in V(G)^2: \{u,v\}\in E(G)} (\E[X_uX_v] - \E[X_u]\E[X_v])$ as $X_u$ and $X_v$ are independent when $\{u,v\}\notin E(G)$. For each $(u,v)$ pair with $\{u,v\}\in E(G)$, we have
$$
\begin{aligned}
\E[X_uX_v] &= \frac{\binom{nd/2-2d+1}{t-1} + (d-1)^2\binom{nd/2-2d+1}{t-2}}{\binom{nd/2}{t}}\\
&= (1+o(1)) (\eps^{2d-2} + d^2 \eps^{2d-1}).
\end{aligned}
$$
Thus, the variance can be bounded as
\[
    \Var[X] \leq \sum_{\substack{(u,v)\in V(G)^2:\\ \{u,v\}\in E(G)}} (1+o(1))(\eps^{2d-2} + d^2 \eps^{2d-1})= nd\eps^{2d-2} + nd^3 \eps^{2d-1} = o(n^2d^2\eps^{2d-2})=o(\E[X]^2)\,.
\]
Therefore, by Chebyshev's inequality, we conclude whp $G_t$ has a vertex of degree $1$.
\end{proof}

We now prove~\Cref{lem:properties-sparse}.
\begin{proof}[Proof of~\Cref{lem:properties-sparse}]
Let $m = e(G_\tau)$. By~\Cref{prop:bound-hitting-time-sparse}, we may assume that $m= (1-\eps)nd/2$ and $0\leq \eps < n^{-1/d}$.
We take a sufficiently large $C$ so that $C > 10^{100}$ holds.
For~\ref{P1}, we first note that as $d \leq 10 \log n$, we have 
$$\binom{d}{d/10^6}\leq (10^6e)^{d/10^6} = e^{ \log(10^6e)/10^6\cdot d}\leq n^{10\log(10^6e)/10^5} \leq n^{2/10^3}\,.$$
Thus, the expectation of $|\SMA(G_\tau)|$ can be bounded as
\begin{align*}
    \E[\SMA(G_{\tau})]  &= n\cdot \sum_{i=0}^{d/10^6} \binom{d}{i}\frac{\binom{nd/2 - d}{m-i}}{\binom{nd/2}{m}}\\
    &\leq n\cdot \frac{d}{10^6}\cdot \binom{d}{d/10^6}\frac{\binom{nd/2-d}{m-d/10^6}}{\binom{nd/2}{m}} \\
    & \leq n\cdot \frac{d}{10^6} \cdot n^{2/10^3} \cdot \frac{(nd/2-d)!}{(nd/2-d-m+d/10^6)!(m-d/10^6)!}\cdot \frac{(nd/2-m)!m!}{(nd/2)!} \\
    &\leq (1+o(1))n\cdot \frac{d}{10^6} \cdot n^{2/10^3} \cdot (1-\eps)^{ d/10^6}\cdot \eps^{d-d/10^6} \leq n^{3/10^3}
\end{align*}
By Markov's inequality, $\Pr(|\SMA(G_\tau)|\geq n^{0.11})\leq n^{3/10^3 - 0.11} = o(1)$, which proves~\ref{P1}.

For~\ref{P2}, note that the probability there exists a vertex $u\in \SMA(G_{\tau})$ for which there is another vertex $v\in \SMA(G_{\tau})$ and $d_{G_\tau}(u,v)\leq 4$ is at most 
\begin{multline*}
    n\sum_{i=1}^4 d^i \binom{d}{d/10^6}^2\cdot \frac{\binom{nd/2 - 2d+1}{m-2\cdot d/10^6}}{\binom{nd/2}{m}}\leq  
    (4+o(1))nd^4 \binom{d}{d/10^6}^2 \cdot \frac{m^{2d/10^6}(nd/2-m)^{2d-2d/10^6}}{(nd/2)^{2d-1}}\\
    < 10d^3\binom{d}{d/10^6}^2 (1-\eps)^{2d/10^6}n^{-1 + 2/10^6} = o(1)\,.
\end{multline*}
by using $\binom{d}{d/10^6}^2\leq n^{4/10^3}\,$.

For~\ref{P3}, it suffices to prove for $|U| =|V| = n/(4C_1^3)$.
By~\Cref{lem:expander-mixing}, we know that
\[
e_{G}(U,V) \geq \frac{d}{n}|U||V| - \lambda \sqrt{|U||V|} \geq 0.99\cdot  \frac{d}{n}\left(\frac{n}{4C_1^3}\right)^2\geq \frac{nd}{20C_1^6}\,,
\]
where we used the fact that $C_1=C^{1/10}$, $C>10^{10}$ and $\lambda/d \leq 1/C$.
If $nd/(20C_1^6) > \varepsilon \cdot nd/2$, then \ref{P3} holds trivially since $G_\tau$ has $m=(1-\varepsilon)nd/2$ edges. 
Otherwise, the probability $e_{G_\tau}(U,V) = 0$ is at most 
\[
    \frac{\binom{nd/2 - nd/(20C_1^6)}{m}}{\binom{nd/2}{m}} = \frac{(nd/2-m)\cdots(nd/2-m-nd/(20C_1^6)+1)}{(nd/2)\cdots(nd/2-nd/(20C_1^6)+1)} \leq (1+o(1))\eps^{nd/(20C_1^6)} \leq n^{-n/(20C_1^6)}\,.
\]
Applying union bound over at most $2^{2n}$ choices of $(U,V)$, we conclude whp there are no $U,V$ satisfying~\ref{P3}, as stated.

For~\ref{P4}, suppose for contradiction there exists a set $W\subseteq V(G)\setminus \SMA(G_\tau)$ such that $|N_{G_\tau}(W)|< C_1|W|$. We divide the proof into two cases based on $|W|$.
Let $W^+ \coloneqq W\cup N_{G_\tau}(W)$. 

If $|W|\leq n/(4C_1^3)$, then by~\Cref{lem:expander-mixing} and $|W^+|\leq (C_1+1)|W|$, we conclude that 
\[
    e_G(W^+)\leq \frac{d}{2n}|W^+|^2 + \frac{\lambda}{2}|W^+| = d|W^+|\cdot \left(\frac{|W^+|}{2n} + \frac{\lambda}{2d}\right) \leq d|W|/C_1\,.
\]
Indeed, note that $|W^+|/2n \leq (C_1+1)/(4C_1^3)\leq 1/(3C_1^2)$ and $\lambda/(2d) \leq 1/(2C) \leq 1/(3C_1^2)$ since $C_1=C^{1/10}$ and $C\ge 10^{100}$.
However, because all edges incident to $v\in W$ in $G_\tau$ belong to $e_{G_\tau}(W\cup N_{G_\tau}(W))$, we conclude
\[
e_{G_\tau}(W\cup N_{G_\tau}(W))\geq \frac{1}{2}\sum_{v\in W}\deg_{G_\tau}(v) \geq \frac{d|W|}{2\cdot 10^6}\,.
\]
Since $G_\tau$ is a subgraph of $G$, it contradicts the assumption that $C_1=C^{1/10}>  10^{10}$.

If $n/(4C_1^3)\leq |W|\leq n/(2C_1)$, then by~\ref{P3}, we know that $|N_{G_\tau}(W)|\geq n-|W| - n/(4C^3)$. Indeed, suppose not, then there exists a set $U$ of size at least $n/(4C^3)$ disjoint from $W$ and $e_{G_\tau}(U,W) = 0$, which contradicts~\ref{P3}. Thus, 
\[
|N_{G_\tau}(W)|\geq n- n/(2C_1) - n/(4C_1^3) \geq n/2 \geq C_1|W|\,,
\]
as desired.
\end{proof}

\subsection{Dense case}
For the sparse case, the computations were done with respect to a uniformly random chosen $m$-edge subgraph of $G$. 
For the dense case, it is more convenient to do the computations with respect to a $p$-subgraph, where each edge of $G$ is included with probability $p$ independently. We will utilize the following result:
\begin{proposition}\label{prop:bound-hitting-time-dense}
     Let $G$ be a $d$-regular $n$-vertex graph, then the following holds:
    \begin{enumerate}[label=\rm{(A\arabic*)}, itemsep = 3pt]
            \item If $p\ge 0$ is such that $n((1-p)^d+dp(1-p)^{d-1})\rightarrow 0$ as $n\rightarrow\infty$, then $G_{p}$ contains no vertex with degree at most $1$ whp. \label{A1}
            \item If $p\ge 0$ is such that $n((1-p)^d+dp(1-p)^{d-1})\rightarrow \infty$ as $n\rightarrow\infty$, then $G_{p}$ contains a vertex with degree at most $1$ whp. \label{A2}
    \end{enumerate}
\end{proposition}

\begin{proof}
Let $X_v$ be the indicator random variable for $\deg_{G_p}(v) \le 1$ and $X = \sum_{v\in V(G)}X_v$.  We have
        $\E[X]=  n((1-p)^d+dp(1-p)^{d-1})$.
        Hence,~\ref{A1} holds by Markov's inequality.
        
For~\ref{A2}, note that 
       $$
       \begin{aligned}
        \mathrm{Var}[X] &= 2\sum_{(u,v)\in \binom{V(G)}{2}}(\mathbb{E}[X_u X_v]-\mathbb{E}[X_u]\mathbb{E}[X_v]) = 2\sum_{\{u,v\}\in E(G)}(\mathbb{E}[X_u X_v]-\mathbb{E}[X_u]\mathbb{E}[X_v])\\
        &= nd( (1-p)^{2d-1} + (2d-1)p(1-p)^{2d-2}
        +(d-1)^2p^2(1-p)^{2d-3})\\
        &\quad\quad-nd((1-p)^d + dp(1-p)^{d-1})^2\\
        &=nd(d-1)^2 p^3 (1-p)^{2d-3}\,.
       \end{aligned}
       $$
        Our goal now is to show that $\Var[X] = o(\E[X]^2)$ and then apply Chebyshev's inequality to conclude. 
        There are two cases: if the first term in $\E[X]$ dominates, which means $dp\leq 1-p$, then it suffices to show $nd(d-1)^2 p^3 (1-p)^{2d-3} = o(n^2 (1-p)^{2d})$, which simplifies to $n (1-p)^3/(dp)^3 = \omega(1)$. This is clearly true given that $dp \leq 1-p$. Otherwise, we have $dp\geq 1-p$ and we need to show that $nd(d-1)^2 p^3 (1-p)^{2d-3} = o(n^2 d^2p^2 (1-p)^{2d-2})$, which simplifies to $n(1-p)/(dp) = \omega(1)$. By the assumption of~\ref{A2} and $dp\geq 1-p$, we know that $ndp(1-p)^{d-1} = \omega(1)$. 
        Observe that $d^2 p^2 (1-p)^{d-2} = O(1)$ for all choices of $d = d(n)$ and $p = p(n)$ for which $d \geq 1$ and $p\in (0,1)$. This is immediate when $d = O(1)$, so we assume $d = \omega(1)$. By taking the derivative, we know the maximizer $p^* = 2/d$ and thus the full expression is at most $4/e^2$, as stated. 
        Putting everything together, we have $n = \omega(1/(dp(1-p)^{d-1})) = \omega(dp/(1-p))$, as desired.
\end{proof}
Next, we establish a bound on the hitting time for minimum degree $2$.
\begin{proposition}\label{prop:bound-hitting-time-new}
     Let $G$ be a $d$-regular $n$-vertex graph.  If $d\geq 10\log n$, then whp, we have $p_1\cdot e(G)\leq \tau_2(G)\leq p_2\cdot e(G)$ where $p_1 \coloneqq 1-e^{-\log n/d}$ and $p_2 \coloneqq (1+10^{-10^{10}})(1-e^{-\log n/d})$.
\end{proposition}
\begin{proof}
    To show the left side of the statement,  using $d>10\log n$ and the second estimate from~\Cref{subsec:estimates}, we have
\[
    n(1-p_1)^d+n\cdot d p_1(1-p_1)^{d-1} \geq n d\left(\frac{\log n}{d}-\frac{1}{2}\frac{\log^2 n}{d^2}\right) e^{-\log n}\ge \frac{1}{2}\log n\,.
\]
By~\Cref{prop:bound-hitting-time-dense}~\ref{A2}, we conclude whp, $G_{p_1}$ contains a vertex of degree at most $1$, and so $p_1\cdot e(G)\le \tau_2(G)$.

    For the right side, if $d=\Theta(\log n)$, let $d=c\log n$ with $c>10$. Then we can write $p_2=1-e^{-1/c-\varepsilon}$, where $\varepsilon$ is a constant depending on $c$. Note that $p_2\le 1/2$ when $c\ge 10$. Then we see that 
    \begin{align*}
        n(1-p_2)^d+dp_2(1-p_2)^{d-1}&\le n(1+d)(1-p_2)^{d}
        \\&= n(1+d)\cdot\exp(d/c-\varepsilon d)
        \\&=(1+c\log n)\exp(-\varepsilon c\log n)=o(1).
    \end{align*}

    If $d=\omega(\log n)$, let $\varepsilon=10^{-10^{10}}>0$, we see that 
    \begin{align*}
        n((1-p_2)^d+dp_2(1-p_2)^{d-1})&\le n(1+d\frac{(1+\varepsilon)\log n}{d})(1-\frac{(1+\varepsilon)\log n}{d}+\frac{(1+\varepsilon)\log^2 n}{2d^2})^{d-1}\\
        &\le n(1+2\log n)(\exp(-(1+\varepsilon-o(1)) \log n))\\
        &=(1+2\log n)n^{-(\varepsilon-o(1))}=o(1)       
    \end{align*}
    The first inequality holds as $$1-\frac{(1+\varepsilon)\log n}{d}\le 1-p_2 \le 1-\frac{(1+\varepsilon)\log n}{d}+\frac{(1+\varepsilon)\log^2 n}{2d^2}.$$
    The second inequality holds since $1-x\le \exp(-x)$. By~\Cref{prop:bound-hitting-time-dense}~\ref{A1}, $G_{p_2}$ contains no vertex with degree at most $1$ whp and so $\tau_2(G)\le p_2e(G)$.
\end{proof}
We note that if $d = \omega(\log n)$, then $1-e^{-\log n/d} =(1+o(1)) \log n/d$. However, this approximation does not work when $d = c\log n$ for some small constant $c>0$.
We provide a general bound that holds for all $d\geq 10\log n$ in the following proposition: 
\begin{proposition}\label{prop:approximation}
    Let $d\geq 10\log n$. Then $(0.95)\log n/d \leq (1-e^{-\log n/d}) \leq \log n/d$.
\end{proposition}
\begin{proof}
    As $1-x \leq e^{-x} \leq 1-x + x^2/2$ for all $x\geq 0$, we have $\log n/d - \frac{1}{2}(\log n/d)^2 \leq 1-e^{-\log n/d} \leq \log n/d$.
    Since $d\geq 10\log n$, we have $\frac{1}{2}(\log n/d)^2 \leq \frac{1}{20}\log n/d$. Thus, we conclude $(1-e^{-\log n/d}) \geq (1 - 1/20)\log n/d = 0.95\log n/d$, as desired.
\end{proof}

Now, we prove that $G_{p_2}$ is ``locally sparse''.
\begin{proposition}\label{prop:locally-sparse}
    Let $C \geq 10^{15}$ and $p \coloneqq (1+10^{-10^{10}})(1-e^{-\log n/d}) $.
    Also, let $G$ be an $(n,d,\lambda)$-graph with $d\geq 10\log n$ and $d/\lambda \geq C$ and $G_p$ be its $p$-random subgraph. Then whp, for any set $S\subseteq V(G)$ of size $s\leq n/(10C)$, $e_{G_p}(S)\leq s\log n/D$ for $D = \log C/10$. 
\end{proposition}
\begin{proof}
        By \Cref{lem:expander-mixing}, we have $e_G(S)\leq \frac{1}{2}(ds^2/n + \lambda s) \leq ds/C$.
        Thus, 
        \begin{align*}
            &\mathbb{P}(\exists S\subseteq V(G): s\leq n/(10C),\, e_{G_{p_2}}(S)\geq s\log n/D)\\
            &\quad \leq \sum_{s = 1}^{n/(10C)}\binom{n}{s} \mathbb{P}\left(\Bin(e_G(S), p) \geq s\log n/D\right)\\
            &\quad \leq \sum_{s = 1}^{n/(10C)}\binom{n}{s} \mathbb{P}(\Bin(ds/C, p)\geq s\log n/D)\,.
        \end{align*}
        Using~\Cref{prop:large-deviation}, $C\geq 10^{15}$ and $D= \log C/10$,
        \begin{multline*}
            \mathbb{P}(\Bin(ds/C, p)\geq s\log n/D) \leq\binom{ds/C}{s\log n/D}p^{s\log n/D}\leq(1.1eD/C)^{s\log n/D} \leq \exp(- 2s\log n)\,.
        \end{multline*}
        Plugging this back in, we conclude 
        \begin{align*}
            &\mathbb{P}(\exists S\subseteq V(G): s\leq n/(10C),\, e_{G_{p_2}}(S)\geq s\log n/D) \\
            &\leq \sum_{s=1}^{n/(10C)} (en/s)^s \exp(-2s\log n) = o(1)\,,
        \end{align*}
        as desired.
\end{proof}

In the remaining section, we prove the following result:
\begin{lemma}\label{lem:properties-dense}
    There exists a constant $C$ such that the following statement holds. 
    Let $G$ be an $(n,d,\lambda)$-graph with $d>10\log n$ and $d/\lambda \geq C$.
    Then the subgraph at the hitting time of minimum degree 2, denoted as $G_\tau$, satisfies the following properties whp. 
    Let $\SMA(G_\tau) \coloneqq \{v\in V(G): \deg_{G_\tau}(v)\leq \log n/10^5\}$ and $C_1=\log^{1/2} C/10^{10}$. Then, 
    \begin{enumerate}[label=\rm{(Q\arabic*)}, itemsep = 3pt]
    \item $\Delta(G_{\tau_2}) \leq 10\log n$; \label{Q1}
    \item $|\SMA(G_{\tau})|\le n^{0.11}$; \label{Q2}
    \item $\forall~u,v\in \SMA(G_{\tau})$, we have $d_{G_{\tau}}(u,v)>4$;\label{Q3}
    \item $\forall~U,V\subseteq V(G)$ disjoint and with $|U|,|V|\ge n/(4C_1^3)$, we have $e_{G_{\tau}}(U,V)\ge 1$; \label{Q4}
    \item $\forall~W\subseteq V(G)\setminus \SMA(G_{\tau})$ with $|W|\le n/(2C_1)$, we have $|N_{G_{\tau}}(W)|\ge C_1|W|$. \label{Q5}
\end{enumerate}
\end{lemma}
\begin{proof}
Let $C>(10e)^{10^{50}}$. Note that $C_1 = \log^{1/2}C/10^{10}\ge 10^{10}$. Let $N\coloneqq e(G)=nd/2$. Let $p_1=1-e^{-\log n/d}$, $p_2=(1+10^{-10^{10}})(1-e^{-\log n/d})$ and $m_1=Np_1$, $m_2=N p_2$. 
For $i\in \{1,2\}$, we write $G_{p_i}$ for the random $p_i$-subgraph of $G$, obtained by including each $e\in E(G)$ independently with probability $p_i$. Similarly, we write $G_{m_i}$ to be the uniformly random $m_i$-edged subgraph of $G$.

From \Cref{prop:bound-hitting-time-new}, we know whp $m_1\le \tau\le m_2$. Note that the natural coupling between the probability measure of $G_{m_1}$, $G_\tau$ and $G_{m_2}$ defined by adding edge one at a time that satisfies $G_{m_1}\subseteq G_{\tau}\subseteq G_{m_2}$ with probability $1$. Therefore, for any increasing property $\mathcal E^+$, we know $\P(G_{m_2}\in \mathcal E^+)\geq \P(G_\tau \in \mathcal E^+) \geq \P(G_{m_1}\in \mathcal E^+)$. 
In addition, by using \Cref{lem:asymptotic-equiv}, we can freely switch the role between $G_{m_i}$ and $G_{p_i}$ for $i\in \{1,2\}$. 
We will frequently make use of these relations to do the computations in  $G_{p_i}$ instead of $G_\tau$.

For~\ref{Q1}, it suffices to show $\Delta(G_{p_2})\leq 10\log n$ whp. For any given vertex $v\in V(G)$, by \Cref{prop:large-deviation}, we know $\mathbb{P}(\deg_{G_{p_2}}(v) \geq 10\log n) = o(n^{-2})$. Applying the union bound over all vertices proves~\ref{Q1}.

For~\ref{Q2}, it suffices to show that $|\text{SMALL}(G_{p_1})|\leq n^{0.11}$ whp. Let $d_0 = \log n/10^5$.
Observe that since $p_1\leq \log n/d$, the expectation of $|\text{SMALL}(G_{p_1})|$ can be bounded as
\begin{align*}
        \mathbb{E}[|\text{SMALL}(G_{p_1})|] &= n\cdot \mathbb{P}(\text{Bin}(d, p_1)\le d_0) \leq n\cdot d_0\cdot \binom{d}{d_0}p_1^{d_0}(1-p_1)^{d-d_0} \\ 
    &\leq  n\cdot d_0\cdot (edp_1/d_0)^{d_0}\cdot \exp\left(-\frac{\log n}{d}(d-d_0)\right)\\&\leq   d_0\cdot n^{10^{-5}\log(10^5 e) + d_0/d} = o(n^{0.09})\,,
\end{align*}
where for the second inequality we used the fact that $1-p_1 = e^{-\log n/d}$.
Thus, by Markov's inequality, we conclude whp $|\text{SMALL}(G_{p_1})|\leq n^{0.11}$. 

For~\ref{Q3}, it suffices to show that whp
in $G_{p_2}$, there is no path of length at most $4$ between any two vertices of $\SMA(G_{p_1})$. 
The proof consists of two steps. 
First, let $S$ be vertices of degree at most $\log n/10^4$ in $G_{p_2}$ (c.f. for $\SMA$ the threshold was $\log n/10^5$). We show that whp $\SMA(G_{p_1})\subseteq S$. Then, we show that whp there is no path of length $4$ between any two vertices in $S$, which implies (the strengthened)~\ref{Q3}.
\begin{claim}\label{cla:small-to-s}
Let $S\coloneqq \{v\in V(G): \deg_{G_{p_2}}(v)\leq \log n/10^4\}$. 
Then, $\SMA(G_{p_1})\subseteq S$ whp.
\end{claim}
\begin{proofclaim}
    Recall that vertices in $\SMA(G_{p_1})$ has degree at most $\log n/10^5$ in $G_{p_1}$. Going from $G_{p_1}$ to $G_{p_2}$, each (missing) edge is added independently with probability $p'\coloneqq (p_2-p_1)/(1-p_1) \leq 10^{-10^9}\log n/d$. Thus, by~\Cref{prop:large-deviation}, the probability there exists a vertex $u\in \SMA(G_{p_1})\setminus S$ is at most 
    \[
        \Pr(\Bin(d, p')\geq \log n/10^4 - \log n/10^5) = o(1/n)\,.
    \]
    By the union bound over all vertices, the proof is finished.
\end{proofclaim}
It remains to prove $\mathbb{P}(\exists u,v\in S: d_{G_{p_2}}(u,v)\leq 4)=o(1)$. We will do the remaining calculations entirely in $G_{p_2}$.
For each vertex $u\in V(G)$, let $N^{(4)}_{G_{p_2}}(u)$ be the set of vertices with distance at most $4$ from $u$ in $G_{p_2}$. 
Observe that
\begin{align*}
    &\mathbb{P}(\exists u,v\in S: d_{G_{p_2}}(u,v)\leq 4) \\
    &\quad \leq \sum_{u\in V(G)}\sum_{v\in N^{(4)}_{G_{p_2}}(u)} \mathbb{P}(u,v\in S)\\
    &\quad \leq \sum_{u\in V(G)}\sum_{v\in N^{(4)}_{G_{p_2}}(u)} \Pr(\Bin(d, p_2)\leq \log n/10^4 - 1)^2\,.
\end{align*}
By~\Cref{prop:approximation} and a direct computation, $\Pr(\Bin(d, p_2)\leq \log n/10^4 - 1) = o(n^{-0.9})$.
Using~\ref{Q1}, we have $\mathbb{P}(\exists u:|N^{(4)}_{G_{m_2}}(u)|\geq 10^5(\log n)^5) = o(1)$. Indeed, since $\Delta(G_{m_2})\leq 10\log n$, we know $ |N^{(4)}_{G_{m_2}}(u)|\leq \sum_{i=1}^4 (10\log n)^i\leq 10^5(\log n)^5$. Assuming this event, we get 
\[
    \mathbb{P}(\exists u,v\in S: d_{G_{p_2}}(u,v)\leq 4) \leq (10^5+o(1))n \log^5 n\cdot n^{-1.8}  = o(1)\,,
\]
as desired.

For~\ref{Q4}, it suffices to prove whp~\ref{Q4} holds with $G_{\tau}$ replaced by $G_{p_1}$ and $|U|=|V|=n/4C_1^3$. 
By~\Cref{lem:expander-mixing}, we know
$$e_{G}(U,V)\ge\frac{d}{n}|U||V|-\lambda\sqrt{|U||V|} \geq \frac{d}{n}\left(\frac{n}{4C_1^3}\right)^2 - \lambda\frac{n}{4C^3_1} \geq \frac{nd}{20C_1^6}\,.$$
By a direct computation and the union bound, 
\begin{align*}
     &\mathbb{P}(\exists\text{ disjoint } U,V\subseteq V(G) \text{ with } |U|,|V|= n/(4C_1^3),\,e_{G_{p_1}}(U,V)=0)\\
     &\quad \leq \binom{n}{|U|}\binom{n}{|V|} \cdot (1-p_1)^{nd/(20C_1^6)}\\
     &\quad \leq (4eC_1^2)^{2n/(4C_1^2)}\cdot e^{-n\log n/(20C_1^6)}= \exp(-\Theta(n\log n)) = o(1)\,,
\end{align*}
as desired.

We now prove~\ref{Q5}. It suffices to show that for $W\subseteq V(G)\setminus \text{SMALL}(G_{\tau})$ and $|W|\leq n/(2C_1)$, we have $|N_{G_{p_1}}(W)|\geq C_1 |W|$ whp. 
We divide the proof into four cases based on the cardinality of $W$:
\begin{description}
    \item[Case 1 ($|W|\leq\log n/(2 \cdot 10^5C_1)$)] Since $W\subseteq V(G)\setminus\text{SMALL}(G_{\tau})$, we know that $\deg_{G_{\tau}}(v) \geq \log n/10^5$ for all $v\in W$ and thus $|N_{G_{\tau}}(W)| \geq |N_{G_{\tau}}(v)| - |W| \geq \log n/(2 \cdot 10^5) \geq C_1|W|$, as desired.
    \item[Case 2 ($\log n/(2 \cdot 10^5C_1)\leq |W|\leq 2C_1n/\log n$)]
    Assume for contradiction that $|N_{G_{\tau}}(W)| < C_1|W|$. Let $W^+\coloneqq W\cup N_{G_{\tau}}(W)$. Then, we have $|W^+|< (C_1+1)|W|\leq 2C_1(C_1+1)n/\log n$. By~\Cref{prop:locally-sparse}, we have 
    $$
    e_{G_{\tau}}(W^+) \leq e_{G_{p_2}}(W\cup N_{G_{p_2}}(W))\leq  \frac{C_1+1}{10^6C_1} \log n \cdot |W|\,.
    $$
    However, because $W\subseteq V(G)\setminus \text{SMALL}(G_{\tau})$ and for every vertex $v\in W$, we have $N_{G_{\tau}}(v)\subseteq W^+$, we conclude 
    $$
    \begin{aligned}
    e_{G_{\tau}}(W^+)&\geq \frac{1}{2}\sum_{v\in X}\deg_{G_{p_1}}(x)
    \geq \frac{1}{2\cdot10^5}\log n \cdot |W|> \frac{C_1+1}{10^6C_1} \log n \cdot |W|\,,
    \end{aligned}
    $$
    which is a contradiction.
    
    \item[Case 3 ($2C_1n/\log n\leq |W|\leq n/(2C_1^2)$)] 
    Assume for contradiction that there exists a set $W\subseteq V(G)\setminus \text{SMALL}(G_{\tau})$ with $2C_1 n/\log n\leq |W|\leq n/(2C_1^2)$ such that $|N_{G_{\tau}}(W)|<C_1|W|$.
    Let $W'\subseteq V(G)$ be an arbitrary set of size $C_1|W|$ such that $W'\supseteq N_{G_{\tau}}(W)$ and $W'\cap W=\emptyset$.
    Let $W^+\coloneqq W\cup W'$ and so $|W^+|=(C_1+1)|W|$.
    By \Cref{lem:expander-mixing} and the range of $|W|$, we know that 
    $$
    e_G(W,W^+) \leq d|W||W^+|/n + \lambda \sqrt{|W||W^+|} \leq d|W|/C_1\,.
    $$
    Thus, $e_{G_{\tau}}(W,W^+)$ is stochastically dominated by $\Bin(d|W|/C_1, p_2)$. 
    Moreover, observe that because $W'\supseteq N(W)$ and $W\subseteq V(G)\setminus \text{SMALL}(G_{\tau})$, we have (deterministically) 
    $$
    \begin{aligned}
    e_{G_{\tau}}(W,W^+) \geq \frac{1}{2}\sum_{v\in W}\text{deg}_{G_{p_1}}(v)\geq |W|\log n/(2\cdot10^5)\,.
    \end{aligned}
    $$
    Thus, by \Cref{prop:large-deviation} and ~\Cref{prop:approximation}, $p_2\le (1+10^{-10^{10}})\log n/d$ and we conclude
    \begin{align*}
         &\mathbb{P}(\exists W\subseteq V(G)\setminus \text{SMALL}(G_{\tau}): 2C_1n/\log n\leq |W|\leq n/(2C_1^2),\,|N_{G_{\tau}}(W)|< C_1|W|) \\
         &\quad \leq \sum_{s = 2C_1n/\log n}^{n/2C_1^2} \binom{n}{s}\binom{n-s}{C_1s} \mathbb{P}\left(\Bin(ds/C_1,p_2)\geq \frac{s\log n}{2\cdot10^5}\right)\\
         &\quad \leq \sum_{s = 2C_1n/\log n}^{n/2C_1^2} 2^{2n}\cdot \binom{ds/C_1}{s\log n/(2\cdot10^5)} p_2^{s\log n/(2\cdot10^5)}\\
         &\quad \leq \sum_{s = 2C_1n/\log n}^{n/2C_1^2} 2^{2n}\cdot (3\cdot10^5e/C_1)^{s\log n/(2\cdot10^5)}\\
         &\quad \leq n\cdot 2^{2n}\cdot (3\cdot10^5e/C_1)^{C_1n/10^5} = o(1),
    \end{align*}
    The final equality holds as $C_1\ge 10^{10}$. Thus, we get the desired conclusion.
    
\item[Case 4 ($n/(2C_1^2)\leq |W|\leq n/(2C_1)$)] 

By~\ref{Q4}, we conclude that $|N_{G_{\tau}}(W)| \geq n/2\ge C_1|W|$. Indeed, if $|N_{G_{\tau}}(W)|<n/2$, then there exists a set $U\subseteq V(G)\setminus(W\cup N_{G_{\tau}}(W))$ of size at least $n/(2C_1^2)$ for which $e_{G_{\tau}}(U,W) = 0$. It contradicts~\ref{Q4}. \qedhere
\end{description}
\end{proof}

\subsection{Proof of \Cref{thm:hitting-time}}\label{subsec:main-proof}
We now show that the properties in \Cref{lem:properties-sparse} and \Cref{lem:properties-dense} are sufficient to apply \Cref{thm:robust-Hamiltonicity} and conclude the proof of \Cref{thm:hitting-time}.
\begin{proof}[Proof of \Cref{thm:hitting-time}]
    Let $C_1 = 10^{10}$ be the constant from \Cref{thm:expander-hamiltonicity} and $C =\exp(C_1 \cdot 10^{11})$.
    Assume that $n$ is sufficiently large compared to $C$.
    Let $G$ be an $(n,d,\lambda)$-graph with $\lambda \leq d/C$.
    We first claim that $G_{\tau_2}$ is close to a $C_1$-expander.
    \begin{claim}
        Whp, $G_{\tau_2}$ contains a set of vertices $X \subseteq V(G)$ with $|X| \leq n^{0.11}$ such that the graph $G_{\tau_2} - X$ is a $C_1$-expander and pairwise distances between vertices in $X$ are at least $5$. 
    \end{claim}
    \begin{proof}
        We separate the proof into two cases depending on the value of $d$.
        If $d \leq 10 \log n$, then by \Cref{lem:properties-sparse}, $G_{\tau_2}$ satisfies \ref{P1}-\ref{P4} whp.
        Let $X = \SMA(G_{\tau_2})$. 
        Then \ref{P1} implies that $|X| \leq n^{0.11}$ and \ref{P2} implies that pairwise distances between vertices in $X$ are at least $5$.
        In addition, \ref{P2} implies for any vertex subset $S \subseteq V(G)$, its neighborhood $N_{G_{\tau_2}}(S)$ contains at most $|S|$ vertices in $X$.
        Thus, $|N_{G_{\tau_2}-X}(S)| \geq |N_{G_{\tau_2}}(S)| - |S| \geq (C^{1/10}-1) \cdot |S| \geq C_1 |S|$ for any $S \subseteq V(G) \setminus X$ with $|S| \leq (n-|X|)/(2C_1) \leq n/(2C_1)$.
        Finally, \ref{P4} imply that for any two vertex disjoint sets $U, V$ of $G_{\tau_2} - X$ of size at least $(n-|X|)/(2C_1) \geq n/(4C_1^3)$, there exists an edge between $U$ and $V$ in $G_{\tau_2}$.
        Therefore, $G_{\tau_2} - X$ is a $C_1$-expander.  

        If $d \geq 10 \log n$, then by \Cref{lem:properties-dense}, $G_{\tau_2}$ satisfies \ref{Q2}-\ref{Q5} whp.
        Let $X = \SMA(G_{\tau_2})$. 
        Then similar to the previous case, \ref{Q2} implies that $|X| \leq n^{0.11}$, \ref{Q3} implies that pairwise distances between vertices in $X$ are at least $5$, and \ref{Q4} and \ref{Q5} imply that $G_{\tau_2} - X$ is a $C_1$-expander.
    \end{proof}
    As $\delta(G_{\tau_2}) \geq 2$, we have that any vertex in $X$ has degree at least $2$ in $G_{\tau_2}$.
    For each vertex $v \in X$, we choose any two neighbors $u, w$ of $v$ in $G_{\tau_2}$, add the edge $uw$ to the set $E_0$, and remove $v$ from the graph. Note that the distance between any two vertices in $X$ is at least $5$, so $u$ and $w$ are not in $X$.
    Let $G'$ be the resulting graph after removing all vertices in $X$ and adding edges as above.
    Then as $G'$ contains $G_{\tau_2}-X$ as a spanning subgraph, $G'$ is also a $C_1$-expander.
    Let $n'$ be the number of vertices in $G'$.
    The set $E_0$ contains at most $|X| \leq n^{0.11} \leq (n')^{0.12}$ edges and all edges of $E_0$ are at least distance $3$ apart in $G'$.
    Therefore, by \Cref{thm:robust-Hamiltonicity}, $G'$ contains a Hamilton cycle containing all edges in $E_0$.
    Then by replacing each edge $uw \in E_0$ with the path of length two containing a vertex in $X$, we obtain a Hamilton cycle in $G_{\tau_2}$.
\end{proof}

\section{Robust Hamiltonicity of $C$-expanders}\label{subsec:robust-hamiltonicity} 

In this section, we prove \Cref{thm:robust-Hamiltonicity}. Before we prove, we remind the statement of \Cref{thm:robust-Hamiltonicity}.
\robust*
\subsection{Embedding a sorting network}
For a graph $H$ and disjoint vertex subsets $A,B \subseteq V(H)$ with $|A| = |B|=m$, we say that $H$ is an \emph{$(A,B)$-linking structure} if for any bijection $\phi: A \to B$, there exists a set of vertex-disjoint paths $P_1, \ldots, P_m$ of equal length such that for each $i \in [m]$, $P_i$ is a path from $a$ to $\phi(a)$ for some $a \in A$ and the union of all paths covers $V(H)$.
The main lemma of this section is the following.
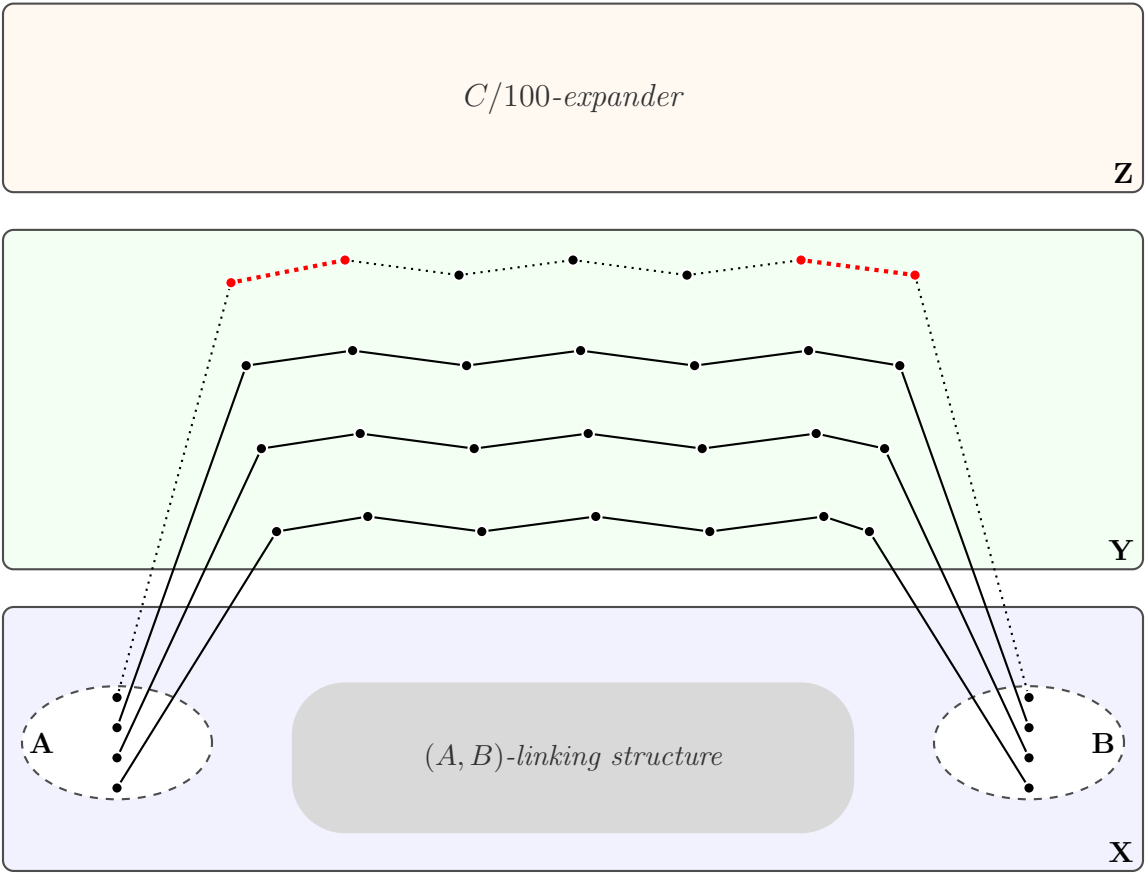
\begin{figure}
    \centering
   \begin{tikzpicture}[
    vertex/.style={circle, fill=black, inner sep=1.5pt, draw=white, thick},
    red_vertex/.style={circle, fill=red, inner sep=1.5pt, draw=white, thick},
    set/.style={rounded corners, thick, draw=black!70},
    subset/.style={ellipse, dashed, thick, draw=black!70, fill=white, minimum width=2.5cm, minimum height=1.5cm},
    edge/.style={thick, black},
    red_edge/.style={ultra thick, red},
    p0_edge/.style={thick, black, dotted},
    p0_red_edge/.style={ultra thick, red, dotted}
]

    \filldraw[fill=blue!5, set] (0, 0) rectangle (15, 3.5);
    \node[above left, font=\large\bfseries] at (15, 0) {X};

    \filldraw[fill=green!5, set] (0, 4) rectangle (15, 8.5);
    \node[above left, font=\large\bfseries] at (15, 4) {Y};

    \filldraw[fill=orange!5, set] (0, 9) rectangle (15, 11.5);
    \node[above left, font=\large\bfseries] at (15, 9) {Z};

    \fill[gray!30, rounded corners=20pt] (3.8, 0.5) rectangle (11.2, 2.5);
    \node[font=\itshape\large, text=black!80, align=center] at (7.5, 1.5) {$(A,B)$-linking structure};

    \node[subset] (setA) at (1.5, 1.7) {};
    \node[font=\bfseries\large, left] at (0.8, 1.7) {A};

    \node[subset] (setB) at (13.5, 1.7) {};
    \node[font=\bfseries\large, right] at (14.2, 1.7) {B};

    
    \node[vertex] (a0) at (1.5, 2.3) {}; 
    \node[vertex] (a1) at (1.5, 1.9) {};
    \node[vertex] (a2) at (1.5, 1.5) {};
    \node[vertex] (a3) at (1.5, 1.1) {};

    \node[vertex] (b0) at (13.5, 2.3) {}; 
    \node[vertex] (b1) at (13.5, 1.9) {};
    \node[vertex] (b2) at (13.5, 1.5) {};
    \node[vertex] (b3) at (13.5, 1.1) {};

    \node[red_vertex] (y0_1) at (3.0, 7.8) {}; 
    \node[red_vertex] (y0_2) at (4.5, 8.1) {};
    \node[vertex] (y0_3) at (6.0, 7.9) {}; 
    \node[vertex] (y0_4) at (7.5, 8.1) {};
    \node[vertex] (y0_5) at (9.0, 7.9) {};
    \node[red_vertex] (y0_6) at (10.5, 8.1) {}; 
    \node[red_vertex] (y0_7) at (12.0, 7.9) {};

    \draw[p0_edge] (a0) -- (y0_1); 
    \draw[p0_red_edge] (y0_1) --  (y0_2); 
    \draw[p0_edge] (y0_2) -- (y0_3) -- (y0_4) -- (y0_5) -- (y0_6); 
    \draw[p0_red_edge] (y0_6) --  (y0_7); 
    \draw[p0_edge] (y0_7) -- (b0);

    \node[vertex] (y1_1) at (3.2, 6.7) {}; \node[vertex] (y1_2) at (4.6, 6.9) {}; 
    \node[vertex] (y1_3) at (6.1, 6.7) {}; \node[vertex] (y1_4) at (7.6, 6.9) {};
    \node[vertex] (y1_5) at (9.1, 6.7) {}; \node[vertex] (y1_6) at (10.6, 6.9) {};
    \node[vertex] (y1_7) at (11.8, 6.7) {};
    
    \draw[edge] (a1) -- (y1_1) -- (y1_2) -- (y1_3) -- (y1_4) -- (y1_5) -- (y1_6) -- (y1_7) -- (b1);

    \node[vertex] (y2_1) at (3.4, 5.6) {}; \node[vertex] (y2_2) at (4.7, 5.8) {}; 
    \node[vertex] (y2_3) at (6.2, 5.6) {}; \node[vertex] (y2_4) at (7.7, 5.8) {};
    \node[vertex] (y2_5) at (9.2, 5.6) {}; \node[vertex] (y2_6) at (10.7, 5.8) {};
    \node[vertex] (y2_7) at (11.6, 5.6) {};
    
    \draw[edge] (a2) -- (y2_1) -- (y2_2) -- (y2_3) -- (y2_4) -- (y2_5) -- (y2_6) -- (y2_7) -- (b2);

    \node[vertex] (y3_1) at (3.6, 4.5) {}; \node[vertex] (y3_2) at (4.8, 4.7) {}; 
    \node[vertex] (y3_3) at (6.3, 4.5) {}; \node[vertex] (y3_4) at (7.8, 4.7) {};
    \node[vertex] (y3_5) at (9.3, 4.5) {}; \node[vertex] (y3_6) at (10.8, 4.7) {};
    \node[vertex] (y3_7) at (11.4, 4.5) {};
    
    \draw[edge] (a3) -- (y3_1) -- (y3_2) -- (y3_3) -- (y3_4) -- (y3_5) -- (y3_6) -- (y3_7) -- (b3);

    \node[font=\Large\itshape, text=black!80] at (7.5, 10.25) {$C/100$-expander};

\end{tikzpicture}

\caption{The structure of the graph $G'$ in \Cref{lem:structure-decomp-new}. Red edges indicate edges of $E_0$, and the dotted path indicates $P_0$. From~\Cref{lem:structure-decomp-new}(5), note that $G'[Z\cup (Y\setminus V(P_0))\cup A\cup B]$ is a $C/100$-expander although not shown on the picture.}
\end{figure}

\begin{lemma}\label{lem:structure-decomp-new}
    Let $C>10^{15}$, let $n$ be sufficient large and $G'$ be a $C$-expander on $n$ vertices with a set of identified edges $E_0 \subseteq E(G')$ of size at most $n^{0.12}$ such that any two edges in $E_0$ are at least distance $3$ apart in $G'$.
    Then, there exists a partition $V(G') = X\cup Y\cup Z$ and two disjoint sets $A,B\subseteq X$ such that the following holds:
    \begin{enumerate}
        \item $|X| \leq n/100$.
        \item $G'[X]$ is an $(A,B)$-linking structure and $|A|=|B|=n^{0.9}$.
        \item $G'[Y\cup A\cup B]$ has a spanning linear forest $\mathcal F$ with paths of length $n^{0.1}/5$ and endpoints in $A\cup B$.
        \item There is a path $P_0\in \mathcal F$ such that $E_0\subseteq P_0$. In particular, $V(E_0) \subseteq Y$.
        \item $G'[Z]$ and $G'[Z\cup (Y\setminus V(P_0))\cup A\cup B]$ are $C/100$-expander.
    \end{enumerate}
\end{lemma}

We apply the Friedman-Pippenger technique~\cite{FP1987} (see also~\cite{DKN2022}) to prove \Cref{lem:structure-decomp-new}.
We first define \emph{extendability}.
\begin{definition}\label{def:extendability}
    Let $D,m$ be positive integers with $D\geq 3$. Let $G$ be a graph and let $H\subseteq G$ be a subgraph with $\Delta(H)\leq D$. Then $H$ is \emph{$(D,m)$-extendable} (in $G$) if for every $S\subseteq V(G)$ with $1\leq |S|\leq 2m$ we have 
    \[
        |\Gamma_G(S) \setminus V(H)| \geq (D-1)|S| - \sum_{u\in S\cap V(H)}(d_H(u)-1)\,.
    \]
\end{definition}
The following lemma allows us to extend the subgraph without destroying the extendability property:
\begin{lemma}[{\cite[Corollary 3.12]{montgomery2019spanning}}]\label{lem:extendable}
    Let $C> 10D$ and let $G$ be an $n$-vertex $C$-expander with a $(D,n/2C)$-extendable subgraph $H$ with $|H|\leq n-5nD/C-\ell$ where $\ell \geq \log n$. Then the following holds for all vertices $y\in V(H)$ with $\deg_H(y)\leq D/2$.
    \begin{itemize}[itemsep = 0.3em]
        \item There exists a path $P$ in $G$ with endpoints $y$ of length $\ell$ such that all its vertices except for $y$ lie outside of $H$ and $H\cup P$ is $(D,n/2C)$-extendable.
        \item For every $x\in V(H)\setminus\{y\}$ with $\deg_H(x)\leq D/2$, there exists an $xy$-path $P$ in $G$ of length $\ell$ such that all its internal vertices lie ouside of $H$ and $H\cup P$ is $(D,n/2C)$-extendable.
    \end{itemize}
\end{lemma} 
To keep applying \Cref{lem:extendable}, the graph needs to have a bounded degree. This motivate the definition of $T$-path-constructible introduced by Hyde, Morrison, M\"uyesser and Pavez-Sign\'{e}~\cite{Sorting-network}:
\begin{definition}[{\cite[Definition 4.2]{Sorting-network}}]\label{def:path-constructible}
    Let $G$ be a graph, $T\subseteq V(G)$ and $1\leq \ell_1\leq l_2$ be integers. We say $G$ is $T$-path-constructible with paths of length between $\ell_1$ and $\ell_2$ if there exists a sequence of paths $P_1,\dots, P_t$ of length between $\ell_1$ and $\ell_2$ satisfying that:
    \begin{enumerate}
        \item $E(G) = \cup_{j\in [t]}E(P_j)$.
        \item For each $i\in [t]$, the internal vertices of $P_i$ are disjoint from $A\cup \bigcup_{j\in [i-1]}V(P_j)$.
        \item For each $i\in [t]$, at least one of the endpoints of $P_i$ belongs to $A\cup \bigcup_{j\in[i-1]}V(P_j)$.
    \end{enumerate}
\end{definition}

Hyde, Morrison, M\"uyesser and Pavez-Sign\'{e}~\cite{Sorting-network} shows there exists an $(A,B)$-linking structure with certain properties, while we use the following weaker version of their result, which is sufficient for our purpose.
\begin{lemma}[{\cite[Lemma 5.5]{Draganic-et-al}}]\label{lem:exist-linking-structure}
    For all sufficiently large $N$, there exists a graph $H$ with at most $N^{1.1}$ vertices and disjoint sets $A,B\subseteq V(H)$ such that $H$ is an $(A,B)$-linking structure and the following holds:
    \begin{enumerate}
        \item $|A| = |B| = N$ and $A\cup B$ is independent set in $H$.
        \item $\Delta(H) \leq 4$.
        \item $H$ is $(A\cup B)$-path-constructible with paths of length between $10\log N$ and $40 \log N$.
    \end{enumerate}
\end{lemma}

We are now ready to prove \Cref{lem:structure-decomp-new}.
\begin{proof}[Proof of \Cref{lem:structure-decomp-new}]
    Let $H$ be the graph obtained from \Cref{lem:exist-linking-structure} with $N = n^{0.9}$ and let $A_H ,B_H \subseteq V(H)$ be the corresponding sets.
    By \Cref{lem:exist-linking-structure}, we have $|V(H)| \leq n/100$, $\Delta(H) \leq 4$, $A_H \cup B_H$ is an independent set in $H$ and $H$ is $(A_H \cup B_H)$-path-constructible with paths of length between $10\log n$ and $40\log n$.

    Without loss of generality, we may assume that $E_0 \neq \emptyset$.
    As edges in $E_0$ are at least distance $3$ apart in $G'$, every vertex $v \in V(G')$ has at most two neighbors in $V(E_0)$.
    Thus, for every $S \subseteq V(G')$ with $|S| \leq n/2C$, we have $|N_{G'}(S) \setminus V(E_0)| \geq |N_{G'}(S)| - 2|S| \geq (C-2)|S| \geq (C/2-1)|S|$. 
    Thus, if we view $E_0$ as a subgraph on $V(E_0)$ with the edge set $E_0$, then it is a $(C/10, n/2C)$-extendable subgraph of $G'$.
    We now choose an arbitrary edge $xy \in E(G') \setminus V(E_0)$.
    From \Cref{def:extendability}, adding the edge $xy$ to $E_0$ results in a $(C/20, n/2C)$-extendable subgraph of $G'$.
    By applying \Cref{lem:extendable}, we can find a path $P$ of length $|A|+|B|$ such that $P$ does not meets $V(E_0)$ and $P \cup E_0$ is a $(C/20, n/2C)$-extendable subgraph of $G'$.
    As $\Delta(P) =2$, it implies for every set $S \subseteq V(G')$ with $|S| \leq n/C$, we have $|N_{G'}(S) \setminus (V(P) \cup V(E_0))| \geq (C/20-1)|S| - \sum_{u \in S \cap (V(P) \cup V(E_0))} (d_H(u)-1) \geq (C/20-2)|S|$.
    Thus, the empty graph on $V(P)$ union with $E_0$ is a $(C/50, n/2C)$-extendable subgraph of $G'$.

    We partition $V(P)$ arbitrarily into two disjoint sets $A$ and $B$ of equal size.
    We view the empty graph on $A \cup B$ as an embedding of $H[A_H \cup B_H]$ into $G''$ with $A_H$ mapped to $A$ and $B_H$ mapped to $B$.
    Since $H$ is $(A_H\cup B_H)$-path-constructible, there exists a sequence of paths $P_1,\dots, P_t$ of length between $10\log n$ and $40\log n$ such that $E(H) = \cup_{j\in [t]}E(P_j)$. Moreover, if we let $H_0 = H[A_H\cup B_H]$ and $H_i = H_{i-1}\cup P_i$ for all $1\leq i\leq t$, then for every $1\leq i\leq t$, the internal vertices of $P_i$ does not intersect $V(H_{i-1})$ and at least one endpoint of $P_i$ belongs to $V(H_{i-1})$. 
    Therefore, if $H_{i-1}$ is embedded into $G' \setminus V(E_0)$ so that the image together with $E_0$ is a $(C/200, n/2C)$-extendable subgraph, then by applying \Cref{lem:extendable}, we can extend the embedding to $H_i$ such that $H_i$ together with $E_0$ is also a $(C/200, n/2C)$-extendable subgraph of $G'$.
    Note that for each $i \in [t]$, $|V(H_i)| + |V(E_0)| \leq n/100 + 2n^{0.12} \leq n - n/10 - 40\log n$, thus \Cref{lem:extendable} can be applied.
    Therefore, by using induction, we can embed $H$ into $G' \setminus V(E_0)$ so that the image together with $E_0$ is a $(C/50, n/2C)$-extendable subgraph.
    Let $\varphi$ be the embedding and let $X = \varphi(V(H))$.

    We now aim to find paths of length $n^{0.1}/5$ connecting $A$ and $B$. 
    Recall that we need a special path $P_0$ for which $E_0\subseteq E(P_0)$, which we will find first. 
    We note that as $\varphi(H) \cup E_0$ is a $(C/50, n/2C)$-extendable subgraph of $G'$, 
    $\varphi(H)$ is also a $(C/50, n/2C)$-extendable subgraph of $G'$.
    
    Let $u_1w_1,u_2w_2,\ldots,u_{|E_0|}w_{|E_0|}$ be an enumeration of all edges in $E_0$. 
    We choose arbitrary $(a, b) \in A \times B$ and denote $w_0:=a$ and $u_{|E_0|+1}:=b$.
    By applying \Cref{lem:extendable} $|E_0|+1$ times, we construct the desired path $P_0$ as follows: 
    Suppose that $P^\prime_0,P^\prime_1,\ldots, P^\prime_{i-1}$ have been constructed for some $i\in [|E_0|]$.
    Since $|E_0|\le n^{0.12}$, by applying \Cref{lem:extendable}, one can find a path $P^\prime_i$ of length $((n^{0.1}/(5|E_0|))-1)$ with internal vertices in $V(G^\prime)\setminus(X\cup V(S)\cup \bigcup_{j<i}V(P^\prime_j))$ that connect 
    $w_i,u_{i+1}$ for each $i\in[|E_0|]$.
    Furthermore, the graph $\varphi(H) \cup E_0 \cup \bigcup_{j\leq i}P^\prime_j$ is still $(C/50, n/C)$-extendable.
    Let $P_0$ be the concatenation of $P^\prime_0, P^\prime_1, \ldots, P^\prime_{|E_0|}$ and edges in $E_0$. Then $P_0$ is the desired length-$(n^{0.1}/5)$ path. 

    As $\varphi(H) \cup P_0$ is still $(C/50, n/C)$-extendable, we apply \Cref{lem:extendable} to $A\setminus\{a\}$ and $B\setminus\{b\}$, one may obtain a family $\mathcal{F}^\prime_1$ of length-$(n^{0.1}/5)$ paths with the endpoint set $A\cup B\setminus\{a,b\}$. Set $\mathcal{F}_1=\{P_0\}\cup\mathcal{F}^\prime_1$, $Y=V(\mathcal{F}_1)\setminus(A\cup B)$ and $Z=V(G^\prime)\setminus (X\cup Y)$.
    We now claim that $G'[Z]$ and $G'[Z\cup (Y\setminus V(P_0))\cup A\cup B]$ are $C/100$-expander.
    Let $S \subseteq Z$ be an arbitrary vertex subset with $|S| \leq n/2C$. 
    Then $|N_{G'[Z]}(S)| = |N_{G'}(S) \setminus (X\cup Y)| = |N_{G'}(S) \setminus V(\varphi(H) \cup \mathcal{F}_1)| \geq (C/50-2) \cdot |S|$ as $\varphi(H) \cup \mathcal{F}_1$ is a $(C/50, n/C)$-extendable subgraph of $G'$.
    Similarly, for every $S \subseteq  Z\cup (Y\setminus V(P_0))\cup A\cup B$ with $|S| \leq n/2C$, we have $|N_{G'[Z\cup (Y\setminus V(P_0))\cup A\cup B]}(S)| \geq |N_{G'}(S) \cap Z| \geq (C/50-9) \cdot |S|$ as $\varphi(H) \cup \mathcal{F}_1$ has maximum degree at most $6$.
    In addition, as $G'$ is a $C$-expander, for any two disjoint vertex subsets $U, W \subseteq Z$ ($Z\cup (Y\setminus V(P_0))\cup A\cup B$, resp) of size $n/2C$, there exists an edge between $U$ and $W$.
    Since $|Z| \geq 3n/4$, this implies that $G'[Z]$ and $G'[Z\cup (Y\setminus V(P_0))\cup A\cup B]$ are $C/100$-expander.
\end{proof}

\subsection{Almost spanning linear forest}
Our next step is to find a spanning linear forest in $G' \setminus (X \setminus (A \cup B))$ that contains all the edges of $E_0$ and endpoints are $A \cup B$.
The key lemma is the following:
\begin{lemma}\label{lem:paths-in-remaining}
    Let $G'$ be an $n$-vertex $C$-expander for $C > 10^{10}$ and $E_0$ be a set of edges with $|E_0| \leq n^{0.12}$ such that any two edges in $E_0$ are at least distance $3$ apart in $G'$.
    Let $A, B, X, Y, Z, \mathcal{F}, P_0$ be defined as in \Cref{lem:structure-decomp-new}.
    Then, there exists a spanning linear forest $\mathcal F_0$ in $G' \setminus (X \setminus (A \cup B))$ such that the following holds:
    \begin{itemize}
        \item $P_0 \in \mathcal{F}_0$. In particular, all edges of $E_0$ are covered by $\mathcal{F}_0$.
        \item $\mathrm{End}(\mathcal{F}_0) = A \cup B$.
    \end{itemize} 
\end{lemma}

The first key lemma is to prove that $G'[Z]$ contains a spanning linear forest with few paths.
\begin{lemma}[{\cite[Lemma 4.1]{Draganic-et-al}}]\label{lem:few-paths}
    Let $G$ be an $n$-vertex $C$-expander for $C > 10^{10}$. Then it contains a spanning linear forest with at most $n^{0.8}$ paths and no isolated vertices.
\end{lemma}

The next key lemma is to reduce the number of endpoints of a linear forest via rotations.
\begin{lemma}[{\cite[Lemma 3.7]{Draganic-et-al}}]\label{lem:reduce-endpoint}
    Let $C>10^{10}$, let $G$ be an $n$-vertex $C$-expander and let $\mathcal F$ be a spanning linear forest in $G$ with no isolated vertices and such that at least $0.1n$ vertices of $G$ belong to paths in $\mathcal F$ of lengths between $100$ and $\sqrt{n}$. Let also $x,y\in \mathrm{End}(\mathcal F)$. Then, there is a spanning linear forest $\mathcal F'$ in $G$ with no isolated vertices such that $|E(\mathcal F) \triangle E(\mathcal F')| = O(\log n)$ and $\mathrm{End}(\mathcal F') = \mathrm{End}(\mathcal F)\setminus \{x,y\}$.
\end{lemma}
\begin{proof}[Proof of \Cref{lem:paths-in-remaining}]
Let $\mathcal F_0$ be the linear forest that combines the forest obtained by applying \Cref{lem:few-paths} to $G[Z]$ and the paths connecting $A$ and $B$ in \Cref{lem:structure-decomp-new}. 
Note that $\mathcal F_0$ is spanning in $V(G)\setminus (X\setminus (A\cup B))$, contains no isolated vertex and has $|Z|^{0.8} + n^{0.9}\leq 2n^{0.9}$ paths. Moreover, at least $n/5$ vertices are on paths of length $n^{0.1}/5$. 
Recall that $P_0\in \mathcal{F}_0$ is the path containing all edges in $E_0$.

Let $\widetilde{G}=G[Z \cup (Y \setminus V(P_0))\cup A\cup B]$ and let $\mathcal{F}_1=\mathcal{F}_0\setminus\{P_0\}$. 
Note that $\mathcal{F}_1$ is a spanning linear forest in $\widetilde{G}$ with no isolated vertices.
By \Cref{lem:structure-decomp-new}, $\widetilde{G}$ is a $C/300$-expander with $n' \geq n-n^{0.1} - n/100 \geq \frac{49n}{50}$ vertices.
In addition, as $V(P_0) \le n^{0.1}$, at least $n/5 - n^{0.1} \ge n/6$ vertices of $\widetilde{G}$ belong to paths in $\mathcal{F}_1$ of lengths between $100$ and $\sqrt{|V(\widetilde{G})|}$.
We now iteratively apply \Cref{lem:reduce-endpoint} to $\widetilde{G}$ and $\mathcal{F}_1$.
Suppose that after $i$ applications of \Cref{lem:reduce-endpoint}, we obtain a linear forest $\mathcal{F}_{i}$.
If $|\mathrm{End}(\mathcal{F}_i)|\ge |A\cup B|-1$, then since $|\mathrm{End}(\mathcal{F}_i)|$ is even, there are at least two endpoints outside of $A\cup B$, call them $x$ and $y$.
By applying \Cref{lem:reduce-endpoint} with $x,y$ and $\mathcal{F}:=\mathcal{F}_i$ we obtain a linear forest $\mathcal{F}_{i+1}$ such that $\mathrm{End}(\mathcal{F}_{i+1})=\mathrm{End}(\mathcal{F}_i\setminus\{x,y\})$ and $|E(\mathcal{F}_{i+1}\Delta\mathcal{F}_i)|=O(\log n)$.
Otherwise, we have $\mathrm{End}(\mathcal{F}_i)=A\cup B\setminus V(P_0)$ and we finish the process.
If the process stops after $k$ steps, then we set $\mathcal{F}^\prime=\mathcal{F}_k \cup \{P_0\}$, then $\mathcal{F}^\prime$ is the desired linear forest.
Therefore, it remains to show that one can apply \Cref{lem:reduce-endpoint} for each step in the process.

As $|\text{End}(\mathcal{F}_1)\setminus(A\cup B)|\le 2n^{0.8}$, we can have at most $2n^{0.8}$ steps in this process.
Thus, $|E(\mathcal{F}_{i}\Delta\mathcal{F}_1)|=O(n^{0.8}\log n)$ for each $i$.
This implies that $\mathcal{F}_i$ still contains at least $n^{0.9}-1-O(n^{0.8}\log n)\ge n^{0.9}/2$ paths that in $\mathcal{F}_1$ and connecting $A$ and $B$.
By the construction of \Cref{lem:structure-decomp-new}, each such path has length $n^{0.1}/5$.
Thus, at least $n/10 \geq |V(\widetilde{G})|/10$ vertices of $\widetilde{G}$ belong to paths in $\mathcal{F}_i$ of lengths between $100$ and $\sqrt{|V(\widetilde{G})|}$. 
Therefore, we can apply \Cref{lem:reduce-endpoint} to $\widetilde{G}$ and $\mathcal{F}_i$ at each step, which completes the proof.
\end{proof}

\subsection{Connecting with the linking structure}
We are now ready to finish the proof of \Cref{thm:robust-Hamiltonicity}.
\begin{proof}[Proof of \Cref{thm:robust-Hamiltonicity}]
    We first apply \Cref{lem:structure-decomp-new} to obtain the partition $V(G') = X\cup Y\cup Z$ and the sets $A,B\subseteq X$.
    Next, we apply \Cref{lem:paths-in-remaining} to obtain a spanning linear forest $\mathcal F_0$ of $G' \setminus (X \setminus (A \cup B))$ that containing all edges in $E_0$ and $\mathrm{End}(\mathcal F_0) = A \cup B$.
    Let $P_1, \ldots, P_m$ be the enumeration of paths in $\mathcal F_0$.
    Finally, as $G'[X]$ is an $(A,B)$-linking structure, for the bijection $\phi: A \to B$ defined by mapping $\mathrm{End}(P_i) \cap A$ to $\mathrm{End}(P_{i+1}) \cap B$ for $i \in [m]$ (with $P_{m+1} = P_1$), there exists a set of vertex-disjoint paths $Q_1, \ldots, Q_m$ in $G'[X]$ such that for each $i \in [m]$, $Q_i$ is a path from $\mathrm{End}(P_i) \cap A$ to $\mathrm{End}(P_{i+1}) \cap B$ and the union of all paths covers $X$.
    Therefore, the union of all paths $P_i$ and $Q_i$ forms a Hamilton cycle in $G'$ containing all edges in $E_0$.
\end{proof}

\section{$k$ edge-disjoint Hamilton cycles}\label{subsec:k-edge-disjoint-HC}
The goal of this section is to prove~\Cref{thm:k-edge-disjoint-HC}.
As many computations in this section are quite similar to those in~\Cref{subsec:hitting-time-properties}, we will omit some repeated calculations.

Let $C\ge\exp(10^{15})$ be a constant, $\hat{d}=\min\{d,10\log n\}$, $k\le\hat{d}/(10^{10}C^2)$, and $\varepsilon=n^{-1/\hat{d}}$. 
These parameters are fixed in the rest of the section.
For each $0\le i<k$, denote by $G^{(i)}$ the graph that is obtained from $G_{\tau_{2k}}$ by removing the edges of $i$ edge-disjoint Hamilton cycles. 
Let ${\rm\text{SMALL}}(G^{(i)})=\{v\in V(G^{(i)}):{\rm\text{deg}}_{G^{(i)}}(v)\le\hat{d}/10^6\}$ and ${\rm\text{LARGE}}(G^{(i)})=V(G^{(i)})\setminus{\rm\text{SMALL}}(G^{(i)})$.

Our goal is to establish that the hitting time for a minimum degree of $2k$ coincides with the hitting time for the existence of $k$ edge-disjoint Hamilton cycles whp. 
To achieve this, it suffices to prove that for any $0 \le i < k$, $G^{(i)}$ maintains the following properties:
\begin{enumerate}[label=(H\arabic*),itemsep = 3pt]
    \item $\Delta(G^{(i)})\le \hat{d}$;\label{H1}
    \item $|{\rm\text{SMALL}}(G^{(i)})|\le n^{0.11}$;\label{H2}
    \item for any $u,v\in {\rm\text{SMALL}}(G^{(i)})$, we have $d_{G^{(i)}}(u,v)\geq 5$;\label{H3}
    \item for any vertex set $S\subseteq{\rm\text{LARGE}}(G^{(i)})$ of size at most $n/(2C_1)$, $|N_{G^{(i)}}(S)|\ge C_1|S|$;\label{H4} 
    \item for any two disjoint vertex sets $S,T\subseteq V(G^{(i)})$ of size at least $n/(2C_1)$, $e_{G^{(i)}}(S,T)\ge 1$.\label{H5}
\end{enumerate}
Through the section, we will make frequent use of the estimates in~\Cref{subsec:estimates}. 
As before, we start by estimating $e(G_{\tau_{2k}})$:
\begin{proposition}\label{prop:bound-hitting-time-for-k}
Let $m = e(G_{\tau_{2k}})$, then whp
\begin{itemize}
    \item if $d\le 10\log n$, $(1-\varepsilon)nd/2\le m\le nd/2$;\label{D1}
    \item if $d\ge 10\log n$, $n\log n/2\le m \le (1+10^{-10^{10}})n\log n/2$.\label{D2}
\end{itemize}
\end{proposition}
\begin{proof}
The former directly follows from~\cref{prop:bound-hitting-time-sparse}. Indeed, deterministically $e(G)\geq e(G_{\tau_{2k}}) \geq e(G_{\tau_2})$.
As for the latter, the lower bound follows by combining~\Cref{prop:bound-hitting-time-new} with~\Cref{lem:asymptotic-equiv}.
For the upper bound, it suffices to show that whp the minimum degree is at least $\log n/(10^{10^{11}})$ if $m=(1+10^{-10^{10}})n\log n/2$. By a direct computation, the probability that there exists a vertex with degree at most $\log n/(10^{10^{11}})$ is at most 
\[
\begin{aligned}
n\cdot \sum_{j=0}^{\log n/(10^{10^{11}})}\frac{\binom{d}{j}\binom{nd/2-d}{m-j}}{\binom{nd/2}{m}}&\overset{(\cref{(1)}),(\cref{(6)})}
{\le} 2n\left(\frac{ed}{\log  n/(10^{10^{11}})}\right)^{\frac{\log n}{10^{10^{11}}}}\frac{\binom{nd/2-d}{m-\log n/(10^{10^{11}})}}{\binom{nd/2}{m}}\\
&\overset{(\cref{(5)})}{\le} 2n(10^{10^{11}}e)^{\frac{\log n}{10^{10^{11}}}}\exp\left(-(1+2\cdot10^{10^{10}})\log n\right)=o(1). \qedhere
\end{aligned}
\]
\end{proof}

Previously, for verifying the joinedness property, we needed only one edge between two large disjoint sets, even though there were many edges in between. For reasons that will be clear, we compute the number of edges between two large sets more accurately.
\begin{lemma}\label{lem:logn_expansion_property} 
Assume $G$ is an $(n,d,\lambda)$-graph with $\lambda\le d/C$. 
Whp $e_{G_{\tau_{2k}}}(S,T)\ge \hat{d}|S|/C$ for any disjoint pair $S,T\subseteq V(G)$ of size $n/(C^{\prime})$, where $C^\prime= C^{1/3}$ and $k\le 10^{-10}\hat{d}/C^2$.
\end{lemma}
\begin{proof}
Let $\varepsilon=n^{-1/d}$. Fix two disjoint subsets $S,T\subseteq V(G)$ of size $s=n/C^\prime$. As $\lambda\le d/C$ and $C^{\prime}=C^{1/3}$, applying \cref{lem:expander-mixing}, we have $e_G(S,T)=(1\pm 1/C^{2/3})ds/{C^\prime}$. Thus,
$$
\begin{aligned}
\mathbb{P}&\left(e_{G_{\tau_{2k}}}(S,T)< \hat{d}s/C\right)\le\sum_{j=0}^{\hat{d}|S|/C}\frac{\binom{e_G(S,T)}{j}\binom{nd/2-e_G(S,T)}{m-j}}{\binom{nd/2}{m}}\overset{(\cref{(6)})}{\le} \frac{2\binom{e_G(S,T)}{\hat{d}|S|/C}\binom{nd/2-e_G(S,T)}{m-\hat{d}|S|/C}}{\binom{nd/2}{m}}\\
&\overset{(\cref{(1)}),(\cref{(5)})}{\le}2\left(\frac{e\cdot e_G(S,T)}{\hat{d}|S|/C}\right)^{\hat{d}|S|/C}\left(\frac{m}{nd/2}\right)^{\hat{d}|S|/C}\cdot\exp\left(-\frac{m\cdot 
e_G(S,T)}{10nd}\right)\le\exp\left(-\frac{n\hat{d}}{C}\right)\,,
\end{aligned}
$$
where $s\ge n/(2C^{2/3})$. 
The last inequality holds because $(1-\varepsilon)n\hat{d}/2\le m\le 1.1n\hat{d}/2$ whp.
Thus, we conclude  
$$
\mathbb{P}\left(\exists S,T\subseteq V(G)\text{ s.t. } |S|=|T|=n/(2C^\prime) : e_{G_{\tau_{2k}}}(S,T)< \hat{d}|S|/C\right)\le 2^{2n}\cdot\exp\left(-\frac{n\hat{d}}{C}\right)=o(1).\qedhere
$$
\end{proof}

We now prove an analog of~\Cref{prop:locally-sparse} that works for all $d$.
\begin{lemma}\label{lem:few-edges-inside}
Let $C\geq \exp(10^{15})$, $C_1 = \log C/1000$. 
Assume $G$ is an $(n,d,\lambda)$-graph with $\lambda\le d/C$. 
Then, for $S\subseteq V(G)$ with $|S|\leq n/C^{1/8}$ and $k\le 10^{-10}\hat{d}/C^2$, we have $e_{G_{\tau_{2k}}}(S)< \hat{d}|S|/(11^6C_1)$.
\end{lemma}
\begin{proof}
Let $(1-\varepsilon)n\hat{d}/2\le m\le1.1n\hat{d}/2$ and $S\subseteq V(G)$ be a set of size $s\in[\log n/(11^6C_1), n/C^{1/8}]$. 
By \Cref{lem:expander-mixing}, we have $e_G(S)= \frac{1}{2}(ds^2/n\pm\lambda s) \leq \alpha ds$ for $\alpha =2/C^{1/8}$. 
If $d\le \alpha^{-1}\log n /(11^6C_1)$, then we have $e_{G_{\tau_{2k}}}(S)\le e_G(S)\le\alpha ds \le\hat{d}s/(11^6C_1)$.

Therefore, assume that $d\ge \alpha^{-1}\log n /(11^6C_1)$ and let $j_0=\hat{d}s/(11^6C_1)$. Then, we have 
\begin{align*}
&\mathbb{P}\left(\exists S\subseteq V(G): s\leq \frac{n}{C^{1/8}},\, e_{G_{\tau_{2k}}}(S)\geq\frac{\hat{d}s}{11^6C_1}\right)
\leq \sum_{s =\log n/(11^6C_1)}^{n/C^{1/8}}\binom{n}{s}\sum_{j=j_0}^{e_{G}(S)}\frac{\binom{e_G(S)}{j}\binom{nd/2-e_G(S)}{m-j}}{\binom{nd/2}{m}}\\
&\overset{(\cref{(1)})}{\le}\sum_{s =\log n/(11^6C_1)}^{n/C^{1/8}}\left(\frac{en}{s}\right)^s\sum_{j=j_0}^{e_G(S)}\frac{\binom{e_G(S)}{j_0}\binom{e_G(S)-j_0}{j-j_0}\binom{nd/2-e_G(S)}{m-j_0-(j-j_0)}}{\binom{nd/2}{m}}\\
&=\sum_{s =\log n/(11^6C_1)}^{n/C^{1/8}}\left(\frac{en}{s}\right)^s\frac{\binom{e_G(S)}{j_0}}{\binom{nd/2}{m}}\sum_{k=0}^{e_G(S)-j_0}\binom{e_G(S)-j_0}{k}\binom{nd/2-e_G(S)}{m-j_0-k}\\
&\overset{(\cref{(4)})}{\le}\sum_{s =\log n/(11^6C_1)}^{n/C^{1/8}}\left(\frac{en}{s}\right)^s\frac{\binom{e_G(S)}{j_0}\binom{nd/2-j_0}{m-j_0}}{\binom{nd/2}{m}}\\
&\overset{(\cref{(1)}),(\cref{(5)})}{\le} \sum_{s =\log n/(11^6C_1)}^{n/C^{1/8}}\left(\frac{en}{s}\right)^s(e\cdot11^6C_1\alpha)^{(\hat{d}s/11^6C_1)}=o(1),
\end{align*}
where the last inequality holds as $\alpha=2/C^{1/8}$, $\hat{d}=10\log n$, and $e_G(S)\le\alpha ds$.
\end{proof}

Previously, when verifying the expansion property, we only cared about ``vertex expansion''. The following lemma tells us that the ``edge expansion'' is good.
\begin{lemma}\label{lem:edge-expansion}
For any $k \le 10^{-10}{\hat{d}}/C^2$, the graph $G_{\tau_{2k}}$ satisfies that for any subset $S\subseteq V(G)\setminus\text{\rm SMALL}(G_{\tau_{2k}})$ and $|S|\leq n/2$, then  
\begin{equation}\label{eqn:edge-expansion}
e_{G_{\tau_{2k}}}(S,S^c) \geq 10^{-5}\hat{d}\cdot |S|\,.
\end{equation}
\end{lemma}
\begin{proof}
We divide the proof into two cases depending on the size of $S$. Assume $G$ is an $(n,d,\lambda)$-graph with $\lambda\le d/C$ for some sufficiently large constant $C>0$.

If $|S|\leq n/C^{1/8}$, then by \cref{lem:few-edges-inside}, we have
\[
e_{G_{\tau_{2k}}}(S)\le \hat{d}|S|/(11^6C_1).
\]
Therefore, one may have 
$$
e_{G_{\tau_{2k}}}(S,S^c)\ge \frac{1}{2}\sum_{v\in S}\text{deg}_{G_{\tau_2}}(v)-e_{G_{\tau_{2k}}}(S)\ge 10^{-5}\hat{d}|S|,
$$
as desired.

Otherwise, by \cref{lem:logn_expansion_property}, we have $T=N_{G_{\tau_{2k}}}(S)\setminus{S}$ has size at least $(1-1/C^{1/4})n-|S|\ge n/2$. 
By \cref{lem:expander-mixing}, one has $e_G(S, T)= ds/2\pm ds/\sqrt{C}$. Thus, we have
$$
\begin{aligned}
{\mathbb{P}\left(e_{G_{\tau_{2k}}}(S,T)< \hat{d}|S|/10^5\right)\le\sum_{j=0}^{\hat{d}|S|/10^5}\frac{\binom{e_G(S, T)}{j}\binom{nd/2-e_G(S, T)}{m-j}}{\binom{nd/2}{m}}\overset{(\cref{(6)})}{\le}\frac{\binom{e_G(S,T)}{\hat{d}|S|/10^5}\binom{nd/2-e_G(S,T)}{m-\hat{d}|S|/C}}{\binom{nd/2}{m}}\overset{(\cref{(1)}),(\cref{(5)})}{\le}\exp\left(-\frac{n\hat{d}}{C}\right).
}\end{aligned}
$$
As $m\ge (1-\varepsilon)n\hat{d}/2$ whp, we conclude that 
$$
\mathbb{P}\left(\exists S,T\subseteq V(G): |S|=|T|=n/(2C^\prime)\text{ and }e_{G_{\tau_{2k}}}(S,T)< \hat{d}|S|/10^5\right)\le 2^{2n}\cdot\exp\left(-\frac{n\hat{d}}{C}\right)=o(1). \qedhere
$$
\end{proof}

We now prove~\cref{thm:k-edge-disjoint-HC}. Many parts involve repeating the calculations done above, but for slightly different parameters. To avoid clustering the paper, we decided to sketch the changes rather than reproduce them in detail. 

\begin{proof}[Proof of \cref{thm:k-edge-disjoint-HC}]
Let $C_1 = \log C/1000$ and $C\geq \exp(10^{15})$.
Our goal is to show that $G^{(i)}$ satisfies properties \cref{H1}-\cref{H5} whp for all $0\le i <k$. 

Property \cref{H1} is immediate when $d \leq 10\log n$, so we may assume that $d\ge10\log n$. 
By~\Cref{prop:bound-hitting-time-for-k}, we know $m\le 1.1n\log n/2$. 
Thus, 
$$
\begin{aligned}
\mathbb{P}&\left(\exists v\in V(G):\text{deg}_{G_{\tau_{2k}}}(v)\ge\hat{d}\right)\le n\sum_{j\ge\hat{d}}\binom{d}{j}\frac{\binom{nd/2~-~d}{m-j}}{\binom{nd/2}{m}}\\
&\overset{(\cref{(4)})}{\le} n\binom{d}{\hat{d}}\frac{\binom{nd/2-\hat{d}}{m-\hat{d}}}{\binom{nd/2}{m}}\overset{(\cref{(1)}),(\cref{(5)})}{\le} n\left(\frac{ed}{\hat{d}}\right)^{\hat{d}}\cdot\left(\frac{m}{nd/2}\right)^{\hat{d}}\le n\left(\frac{1.1e}{10}\right)^{10\log n}=o(1),
\end{aligned}
$$
as $\hat{d}=10\log n$, $k \le 10^{-10}{\hat{d}}/C^2$ and $\binom{d}{j}\le\binom{d}{\hat{d}}\binom{d-\hat{d}}{j-\hat{d}}$. 
Thus, whp we have $\Delta({G_{\tau_{2k}}})\le\hat{d}$ for all $k\le 10^{-10}\hat{d}/C^2$. 
For each $i<k$, $G^{(i)}$ has property \cref{H1} directly.

For property \cref{H2}, we have proven it for $k = 1$ (appeared as~\ref{P1} and~\ref{Q2}). Because $\SMA(G_{\tau_{2k}})\subseteq \SMA(G_{\tau_2})$, we know whp $G_{\tau_{2k}}$ also has property
\cref{H2}. 
However, the deletion of $i$ edge disjoint Hamilton cycles may add vertices to $\SMA(G^{(i)})$, thus violating ~\cref{H2}. 
To avoid this, we introduce a vertex set $\MED(G)=\{v\in V(G):\deg(v)\le\hat{d}/10^5\}$ and let $\LAR(G)=V(G)\setminus\MED(G)$. 
Then, we may prove a stronger version of~\ref{P1} (or~\ref{Q2}) with $\SMA(G_{\tau_2})$ replaced by $\MED(G_{\tau_2})$ using the exact same argument. 
By monotonicity, $|\MED(G_{\tau_{2k}})| \leq n^{0.11}$.
Since $k\le \hat{d}/(10^{10}C^2)$ and each Hamilton cycle decreases the degree of every vertex by $2$, we know that $\SMA(G^{(i)})\subseteq \MED(G_{\tau_{2k}})$, which proves~\ref{H2}.

For property~\ref{H3}, we first introduce $\MED(G_{\tau_{2k}})$ as before and show a strengthened version of~\ref{H3} for $\MED(G_{\tau_{2k}})$. 
This will imply~\ref{H3}. Indeed, note that $\SMA(G^{(i)})\subseteq \MED(G_{\tau_{2k}})$ for all $i\leq k$. 
Also, since $G^{(i)}$s are obtained from $G_{\tau_{2k}}$ by removing Hamilton cycles, the distance between two vertices only increases. 
To show the strengthened~\ref{H3}, we divide into sparse and dense cases and repeat the proof of~\ref{P2} and~\ref{Q3}, respectively.

For property \cref{H5}, we need to show that for any two disjoint subsets $S,T\subseteq V(G)$ with $|S|=|T|= n/(2C_1)$, we have $e_{G^{(i)}}(S,T)\geq 1$. 
By~\cref{lem:logn_expansion_property}, we know $e_{G_{\tau_{2k}}}(S,T)> 2k|S|$.
As $k$ edge-disjoint Hamilton cycles use at most $2k|S|$ edges between $S,T$ and $k\le 10^{-10}{\hat{d}}/C^2<\hat{d}/2C$,  we conclude $e_{G^{(i)}}(S,T)\ge 1$ for all $i\leq k$.

It remains to verify property \cref{H4}, i.e., for any subset of vertices $S\subseteq V(G)\setminus\text{\rm SMALL}(G^{(i)})$ with $|S|\leq n/(2C_1)$, we have $|N_{G^{(i)}}(S)|\geq C_1|S|$ for all $i< k$.
We divide the proof into three cases: 
\begin{description}
\item[Case 1 ($|S|\le n/C^2$)] 
Let $T=N_{G^{(i)}}(S)$. 
Assume for contradiction that $|T|< C_1|S|$. 
By ~\cref{lem:expander-mixing}, we have 
$$
e_G(S,T)\leq\frac{d}{n}|S||T|+ \lambda\sqrt{|S||T|}\le \frac{2d\sqrt{C_1}}{C}|S|.
$$
If $d\le\sqrt{C}\log n$, then $e_{G_{\tau_{2k}}}(S,T)\le e_G(S,T)\leq \hat{d}|S|/C_1$.
If $d\ge \sqrt{C}\log n$, then  
\begin{align*}
\mathbb{P}&\left(\exists S : |S|\le n/C^2\text{ and }e_{G_{\tau_{2k}}}(S,T)\ge \hat{d}|S|/C_1\right)\le\sum_{s\le n/C^2}\binom{n}{s}\sum_{j= \hat{d}|S|/C_1}^{e_G(S,T)}\frac{\binom{e_G(S,T)}{j}\binom{nd/2-e_G(S,T)}{m-j}}{\binom{nd/2}{m}}\\
&\overset{(\cref{(4)})}{\le}\sum_{s\le n/C^2}\binom{n}{s}\frac{\binom{e_G(S,T)}{\hat{d}|S|/C_1}\binom{nd/2-\hat{d}|S|/C_1}{m-\hat{d}|S|/C_1}}{\binom{nd/2}{m}}\overset{(\cref{(1)}),(\cref{(5)})}{\le} \sum_{s\le n/C^2}\left(\frac{en}{s}\right)^{s}\left(\frac{eC_1^{2}}{C}\right)^{\hat{d}s/C_1}=o(1),
\end{align*}
as $C=e^{1000C_1}$, $\hat{d}=10\log n$ and $\binom{e_G(S,T)}{j}\le\binom{e_G(S,T)}{\hat{d}|S|/C_1}\binom{e_G(S,T)-\hat{d}|S|/C_1}{j-\hat{d}|S|/C_1}$. 
Thus, for any $d$, we have $e_{G_{\tau_{2k}}}(S,T)\leq \hat{d}|S|/C_1$. However, by~\cref{lem:edge-expansion},
$$
e_{G_{\tau_{2k}}}(S,T)\ge e_{G_{\tau_{2k}}}(S,S^c)-2k|S|\ge10^{-6}\hat{d}|S|>\hat{d}|S|/C_1\,,
$$
which is a contradiction.

\item[Case 2 ($n/C^2\le |S|\le n/C^{1/4})$]
Since $|S|\geq n/C^2$, we have $|N_{G_{\tau_{2k}}}(S)|\le n\leq C^2|S|$. It suffices to show that there exists a set $T\subseteq N_{G_{\tau_{2k}}}(S)$ for which $|T| \geq C_1|S|$ and that $\deg_S(v) \ge 2k$ for all $v\in T$. 
On the one hand, by \cref{lem:few-edges-inside}, we know that 
\[
    e_{G_{\tau_{2k}}}(S,T) \leq e_{G_{\tau_{2k}}}(S\cup T) < \frac{\hat{d}(|S|+ |T|)}{11^6 C_1}\,.
\]
On the other hand, as $|N_{G_{\tau_{2k}}}(S)|< C^2|S|$ and $k = 10^{-10}\hat{d}/C^2$, by \cref{lem:edge-expansion}, we conclude
\begin{align*}
    e_{G_{\tau_{2k}}}(S,T)&\geq e_{G_{\tau_{2k}}}(S,N_{G_{\tau_{2k}}}(S)) -e_{G_{\tau_{2k}}}(S,N_{G_{\tau_{2k}}}(S)\setminus T) \\
    &\geq 10^{-5}\hat{d}|S| - 2k\cdot 10^{-10}|S|/C^2 \geq 10^{-6}\hat{d}|S|\,.
\end{align*}
Putting them together, we conclude $|T|\geq C_1|S|$ for a sufficiently large $C_1>0$, as desired.

\item[Case 3 ($n/(C^{1/4})\le|S|\le n/2C_1$)]
By \cref{lem:logn_expansion_property}, we have $|N_{G_{\tau_{2k}}}(S)|\ge(1-1/C^{1/4})n$. 
Let $W=\{v\in N_{G_{\tau_{2k}}}(S)\setminus{S}:\text{deg}_{G_{\tau_{2k}}}(v,S)\le 2k\}$. 
We may assume $|W|\ge n/C^{1/4}$. 
Indeed, otherwise by~\cref{lem:logn_expansion_property}, we have $c\hat{d}|S|\le e_{G_{\tau_{2k}}}(W,S)\le 2k|S|$, a contradiction. 
Thus, we have $|T|\ge n/2\ge C_1|S|$. 
\end{description}
Altogether, we conclude that $G^{(i)}$ has \cref{H1}-\cref{H5} whp.
Thus, by following the same proof as in \cref{subsec:main-proof} to each $G^{(i)}$, we can find one Hamilton cycle in each $G^{(i)}$ and therefore $k$ edge-disjoint Hamilton cycles in $G_{\tau_{2k}}$, as desired.
\end{proof}

\section{Deducing the sharp threshold}\label{subsec:deducing-corollaries}
The goal of this section is to prove~\Cref{cor:sharp-threshold} and~\Cref{cor:sharp-threshold-refine}.
 Note that ~\Cref{cor:sharp-threshold} follows from the ~\Cref{cor:sharp-threshold-refine}, so we will only prove ~\Cref{cor:sharp-threshold-refine}. 
\begin{proof}[Proof of~\Cref{cor:sharp-threshold-refine}]
   Combining~\Cref{thm:hitting-time} and~\Cref{prop:bound-hitting-time-dense}, it suffices to verify the two conditions~\ref{A1} and~\ref{A2} in~\Cref{prop:bound-hitting-time-dense}.

    For the first case, if $d=o(\log n)$, then when $p=1-(nd)^{-(1+\varepsilon)/(d-1)}$ and $n\rightarrow\infty$,
    $$n\left((1-p)^d+dp(1-p)^{d-1}\right)=(1+o(1))(nd)^{-\varepsilon} \rightarrow 0\,,$$ so~\ref{A1} is satisfied.
    If $p=1-(nd)^{-(1-\varepsilon)/(d-1)}$, by a similar computation, ~\ref{A2} is satisfied.

    For the second case, if $p=1-e^{-\log n/d-\varepsilon}$, since $d=c\log n$ for some constant $c>0$, then $p=(1-e^{-1/c-\varepsilon})$. Note that $$n\left((1-p)^d+dp(1-p)^{d-1}\right)=(1+o(1))pdn^{-\varepsilon c} \rightarrow 0\,,$$
    so~\ref{A1} is satisfied.
   If $p=(1-e^{-1/c+\varepsilon})$, by a similar computation, ~\ref{A2} is satisfied.

    For the third case, if $p = (\log n+ \log\log n)/d-(1-\varepsilon)\log^2 n/(2d^2)$, note that $(1-p)^d = \exp( -dp - dp^2/2 + O(dp^3) )$. Hence,  $$n\left((1-p)^d+dp(1-p)^{d-1}\right)=(1+o(1))\exp\left( \frac{-\varepsilon \log^2 n}{2d} \right)\rightarrow 0\,,$$ 
    so~\ref{A1} is satisfied.
    If $p = (\log n+ \log\log n)/d-(1+\varepsilon)\log^2 n/(2d^2)$, by a similar computation, ~\ref{A2} is satisfied.

    For the fourth case, if $p = (\log n+ \log\log n+\omega(1))/d$, note that $(1-p)^d = \exp( -dp - O(dp^2))$. Then 
    $$n\left((1-p)^d+dp(1-p)^{d-1}\right)=(1+o(1))\exp\left( -\omega(1) \right)\rightarrow 0\,,$$ so~\ref{A1} is satisfied.
    If $p=(\log n+ \log\log n-\omega(1))/d$, by a similar computation, ~\ref{A2} is satisfied.
\end{proof}

\section{Concluding remarks}\label{sec:concluding-remarks}
In this paper, we studied the hitting time for Hamiltonicity in $(n,d,\lambda)$-graphs and proved the optimal results. We finish the papers with a few remarks and open problems.

First, we note it is straightforward to generalise~\Cref{thm:k-edge-disjoint-HC} to work for minimum degree $2k+1$ and $k$ Hamilton cycles plus a perfect matching, all edge disjoint, with the same parameters. Indeed, recall the strategy for~\Cref{thm:k-edge-disjoint-HC} was to remove Hamilton cycles one by one and check at each time the resulting graph is still a $C$-expander. Removing a perfect matching instead of a Hamilton cycle in the first round does not affect the rest of the proof.

Second, our results fundamentally rely on the property that a $C$-expander is robustly Hamiltonian. A natural question arises: can the hitting time result be extended to general $C$-expanders? The following example demonstrates that the answer is negative.
\begin{proposition}
    There exists an $\Omega((n/\log n)^{1/2})$-expander $G$ on $n$ vertices such that whp, $G_{\tau_2}$ is not Hamiltonian.
\end{proposition}
\begin{proof}[Proof Sketch]
    Let $c>0$ be a sufficiently small constant, determined later. 
    Let $G_1$ be the complete bipartite graph $K_{n/3, 2n/3}$ with parts $A$ and $B$.
    Let $G_2$ be the $(n, d, \lambda)$-graph on $A \cup B$ with $d = c n/\log n$ and $\lambda \leq 3\sqrt{d}$.
    Note that such a graph exists for sufficiently large $n$ by considering a random $d$-regular graph~\cite{Friedman}.
    Let $G = G_1 \cup G_2$.
    Then as $G_2$ is an $\Omega(d/\lambda) = \Omega((n/\log n)^{1/2})$-expander, $G$ is also an $\Omega((n/\log n)^{1/2})$-expander.
    We now observe that as all vertices have degree at least $n/3$, the hitting time $\tau_2 = \Theta(n \log n)$ whp.
    Thus, by Chernoff's bound, whp $G_{\tau_2}[A]$ and $G_{\tau_2}[B]$ contains at most $2e(G_2) \cdot (\tau_2/e(G))  \leq n/4$ 
    edges by choosing $c$ sufficiently small.
    On the other hand, as $G_1$ is a complete bipartite graph, every Hamilton cycle of $G_{\tau_2}$ must contain at least $n/3$ edges in $G_{\tau_2}[B]$. Therefore, whp $G_{\tau_2}$ is not Hamiltonian.
\end{proof}
It is natural to ask what are sufficient conditions for a $C$-expander to have the hitting time property. For instance, as our construction is far from being regular, we would like to ask whether imposing the graph to be regular suffices.
\begin{question}
    Does there exists a $d$-regular graph $C$-expander $G$ for a sufficiently large constant $C>0$ but does not satisfy $G_{\tau_2
    }$ is Hamiltonian whp?
\end{question}

Third, we observe two distinct regimes governing the structural properties of $G_{\tau}$, which fundamentally alter the complexity of packing edge-disjoint Hamilton cycles.
\begin{enumerate}[label=(\arabic*)]
    \item The Sparse Regime ($k \ll \min\{\log n, d\}$): 
In this regime, the degree distribution of vertices in $G_{\tau}$ converges asymptotically to a Poisson distribution. Consequently, $G_{\tau}$ is structurally highly irregular: while the minimum degree $\delta(G_{\tau})$ is fixed at $2k$, the average degree is significantly higher. 
When employing the edge deletion method to establish the existence of $k$ edge-disjoint Hamilton cycles, this irregularity is advantageous. The removal of edges required for the first few cycles has a negligible impact on the global connectivity; the vast majority of vertices retain a degree much larger than $\delta(G_{\tau})$. Thus, the remaining graph preserves strong expansion properties on the set of large-degree vertices, providing the `structural slack' necessary to close subsequent cycles.
    \item The Dense Regime ($k \gg \log n$): 
Conversely, when $k$ is large, the degree distribution of $G_{\tau}$ is asymptotic to the normal distribution due to the concentration of measure. In this setting, $G_{\tau}$ becomes almost regular, with $\delta(G_{\tau}) \approx \Delta(G_{\tau}) \approx \mathbb{E}[d]$. 
This transition fundamentally changes the nature of the problem. Proving the hitting time for $\lfloor\delta(G_{\tau})/2\rfloor$ edge-disjoint Hamilton cycles shifts from an expansion problem to a tight packing problem. Since the minimum degree is close to the average degree, a maximal packing of cycles requires utilizing nearly every edge in the graph. Unlike the sparse case, there is no `high-degree core' to buffer the edge deletion; the loss of edges degrades the expansion properties uniformly across the graph. This lack of redundancy makes the dense regime significantly more analytically challenging than the sparse case.
\end{enumerate}
Thus, a natural question to ask, which necessarily requires a different approach, is the following:
\begin{question}
    Given an $(n,d,\lambda)$-graph with $d = \omega(\log n)$ and $d/\lambda \geq C$ for some sufficiently large constant $C>0$, is it true that $\tau_{2k} = \tau_{k\HC}$ for all $k\in [\lfloor d/2\rfloor]$ whp?
\end{question}
We would like to remark that when $d$ is even and $k=d/2$, the question is equivalent to asking for a Hamiltonian decomposition of an $(n,d,\lambda)$-graph and the approximate version of this problem was recently solved by Dragani\'{c}, Kim, Lee, Munh\'{a} Correia, Pavez-Sign\'{e} and Sudakov~\cite{Draganic-et-al-resilience}.

\appendix
\section{Hitting time result for random graphs}\label{sec:appendix}
Recall the following asymptotic equivalence between $G(n,m)$ and $G(n,p)$:
\begin{proposition}[{\cite[Proposition 1.2]{janson-book}}]\label{prop:equiv-appendix}
    Let $\mathcal P$ be a graph property and $p = p(n)\in [0,1].$ If for every sequence of $m = m(n)$ with $m = \binom{n}{2}p + O\left(\sqrt{\binom{n}{2}p(1-p)}\right)$, we have $\Pr(G(n,m)\in \mathcal P)\rightarrow a$ as $n\rightarrow \infty$, then also $\Pr(G(n,p)\in \mathcal P) \rightarrow a$ as $n\rightarrow \infty$.
\end{proposition}

The goal of this section is to prove the following:
\begin{proposition}\label{prop:Gnm-appendix}
    Let $m = (n \log n + n\log\log n+\omega(n))/2$ and $G= G(n, m)$ be a graph chosen uniformly at random from all graphs on $n$ vertices and $m$ edges. Then whp $G$ satisfies $\mathbb{P}(\tau_{\HC}(G) = \tau_{2}(G)) = 1-o(1)$.
\end{proposition}
The $G(n,p)$ version follows directly by combining~\Cref{prop:Gnm-appendix} with~\Cref{prop:equiv-appendix}.

We introduce the following notation: For a graph $G$, a permutation $\sigma$ of $E(G)$, and an integer $t\in \mathbb N$, let $G_t(\sigma)$ be the subgraph obtained from $G$ by keeping only edges $\sigma(1),\dots, \sigma(t)$. Also, let $G(\sigma) = (G_0(\sigma),G_1(\sigma),\dots)$ so that we can define the hitting time as before.
At last, to avoid clustered formulas, we write $K$ for $K_n$.
\begin{proof}
    Fix an arbitrary $\eps > 0$ and assume that $n$ is sufficiently large.
    Let $\mathcal G_m$ denote the set of graphs on $n$ vertices with $m$ edges.
    For each $G\in \mathcal G_m$, let $S_G$ be the set of permutations $\sigma$ of $E(K)$ such that $K_m(\sigma)=G$.
    Note that $|S_G| = m!(n-m)!$ for all $G\in \mathcal G_m$ and 
    $\{S_G\}_{G\in \mathcal G_m}$ partitions the set of all permutations of $E(K)$.
    
    Let $S'_G\subseteq S_G$ be the set of $\sigma \in S_G$ such that $\tau_{\HC}(K(\sigma)) > \tau_2(K(\sigma))$.
    Since $\tau_2(K) = \tau_{\HC}(K)$ whp, 
    \begin{equation}\label{eqn:whp}
        \frac{\sum_{G\in \mathcal G_m} |S'_G|}{\sum_{G\in \mathcal G_m} |S_G|} = o(1).
    \end{equation}
    Let $\mathcal H_1$ denote the set of graphs $G\in \mathcal G_m$ such that $|S'_G|/|S_G|\leq \eps$ and $\mathcal{H}_2$ the set of graphs $G\in \mathcal G_m$ such that $G$ is Hamiltonian. 
    For each $G\in \mathcal H_1 \cap \mathcal{H}_2$, since $|S'_G|/|S_G|\leq \eps$, for a uniformly random permutation $\sigma\in S_G$, we have $\Pr_\sigma(\tau_2(G(\sigma)) = \tau_{\HC}(G(\sigma))) \geq 1-\eps$. 
    Let $\pi$ be a uniformly random permutation of $E(G)$.
    Because $\sigma$ restricted to the first $m$ entries has the same distribution as $\pi$, we know $\Pr_\pi(\tau_2(G(\pi)) = \tau_{\HC}(G(\pi))) \geq 1-\eps$.

    Note that $|\mathcal H_1|/|\mathcal G_m| \geq 1-\eps$, otherwise the left-hand side of~\eqref{eqn:whp} is at least $\eps^2$, a contradiction. Since $G(n,m)$ is Hamiltonian whp, we also have $|\mathcal{H}_2|/|\mathcal G_m| \geq 1-\varepsilon$.
    Thus, $|\mathcal H_1 \cap \mathcal{H}_2|/|\mathcal G_m| \geq 1-2\varepsilon$.
    As $\varepsilon>0$ is arbitrary, $\Pr(G(n,m)\in \mathcal H_1 \cap \mathcal{H}_2) = 1-o(1)$. Also, for each $G\in \mathcal H_1 \cap \mathcal{H}_2$, we know $\Pr_\pi(\tau_2(G(\pi)) = \tau_{\HC}(G(\pi))) = 1 - o(1)$ where $\pi$ is a uniformly random permutation of $E(G)$, as desired.
\end{proof}

\end{document}